\documentclass[11pt]{amsart} 

\usepackage[margin=3.5cm]{geometry}
\usepackage{amsmath} 
\usepackage{amssymb} 
\usepackage{amsthm} 
\usepackage{color} 
\usepackage{mathrsfs} 
\usepackage{eucal}
\usepackage{braket} 
\usepackage{stmaryrd}
\usepackage{tikz}
\usepackage{tikz-cd}
\usepackage{enumerate}

\DeclareMathOperator{\sSet}{\mathcal{S}et_{\Delta}}

\DeclareMathOperator{\Sets}{\mathcal{S}et}

\newcommand{\Fun}{\mathrm{Fun}}
\newcommand{\map}{\mathrm{map}}

\newcommand{\Hom}{\mathrm{Hom}} 
\newcommand{\Ho}{\mathrm{Ho}} 
\DeclareMathOperator{\h}{\mathrm{h}}

\DeclareMathOperator{\sd}{\mathrm{sd}}

\DeclareMathOperator{\Tw}{\mathrm{Tw}} 
\DeclareMathOperator{\op}{\mathrm{op}}
\DeclareMathOperator{\id}{\mathrm{id}}

\DeclareMathOperator{\Corr}{\mathcal{C}orr}
\DeclareMathOperator{\boldR}{\mathbf{R}}

\newcommand{\cS}{\mathcal{S}}

\theoremstyle{plain}
\newtheorem*{theorem*}{Theorem}
\newtheorem{theorem}{Theorem}[section]
\newtheorem{thmx}{Theorem}
\newtheorem{lemma}[theorem]{Lemma}
\newtheorem{proposition}[theorem]{Proposition}
\newtheorem{corollary}[theorem]{Corollary}

\theoremstyle{definition}
\newtheorem{definition}[theorem]{Definition}
\newtheorem{remark}[theorem]{Remark} 

\newtheorem{notation}[theorem]{Notation}

\begin{document} 

\title[Model structures for correspondences and bifibrations]{Model structures for correspondences and bifibrations}
  \author[D.\ Stevenson]{Danny Stevenson}
  \address[Danny Stevenson]
  {School of Mathematical Sciences\\
  University of Adelaide\\
  Adelaide, SA 5005 \\
  Australia}
  \email{daniel.stevenson@adelaide.edu.au}

\thanks{This research was supported under the Australian
Research Council's {\sl Discovery Projects} funding scheme (project number DP180100383).}


\begin{abstract}
We study the notion of a bifibration in simplicial sets which 
generalizes the classical notion of two-sided discrete fibration studied in category theory.  
If $A$ and $B$ are simplicial sets we equip the category $(\mathcal{S}\mathrm{et}_{\Delta})_{/(A\times B)}$ of  
simplicial sets over $A\times B$ with the structure of a model category for 
which the fibrant objects are the bifibrations from $A$ to $B$.  
We also equip the category $\mathcal{C}\mathrm{orr}(A,B)$ 
of correspondences of simplicial sets from $A$ to $B$ with the structure 
of a model category.  We describe several Quillen equivalences relating 
these model structures with the covariant model structure on 
$(\mathcal{S}\mathrm{et}_{\Delta})_{/(B^{\mathrm{op}}\times A)}$  
\end{abstract}
\maketitle

\section{Introduction} 
\label{sec:introduction}
A useful concept from ordinary 
category theory is 
the notion of {\em profunctor}.  This 
has several incarnations.  If $A$ 
and $B$ are categories, then a profunctor 
from $A$ to $B$ 
may be viewed as a functor 
$F\colon B^{\op}\times A\to \Sets$, or 
equivalently as a colimit preserving 
functor $\mathcal{P}(A)\to\mathcal{P}(B)$ 
between the categories of presheaves 
on $A$ and $B$ respectively.  There is an 
equivalence of categories 
\begin{equation}
\label{eq:one}
[B^{\op}\times A,\Sets] \xrightarrow{\sim} 
\Corr(A,B) 
\end{equation}
between the category of profunctors from 
$A$ to $B$, and the category $\Corr(A,B)$ 
of {\em correspondences} from $A$ to $B$, i.e.\ 
functors $p\colon C\to [1]$ such that 
$p^{-1}(0) = B$ and $p^{-1}(1) = A$.

There is also an equivalence of categories 
\begin{equation}
\label{eq:two}
\Corr(A,B) \xrightarrow{\sim} 
\mathcal{DF}\mathrm{ib}(A,B) 
\end{equation}
between the category of correspondences from 
$A$ to $B$ and the category $\mathcal{DF}\mathrm{ib}(A,B)$ 
of {\em two-sided discrete fibrations} 
from $A$ to $B$.  A two-sided discrete fibration 
$(p,q)\colon X\to A\times B$ is, roughly speaking, a functor 
whose fibers $X(a,b)$ are covariant in $a\in A$ and 
contravariant in $b\in B$.  The concept was  
exploited by Street in \cite{Street1, Street2}.  

There are analogues for the notions 
of profunctor, correspondence and 
two-sided discrete fibration at the level 
of simplicial sets.  These notions have been 
studied in \cite{AF,BSP,H,HTT,J1,RV1}.  If $A$ and $B$ 
are simplicial sets, then a profunctor 
from $A$ to $B$ may be thought of as a 
simplicial map $B^{\op}\times A\to \cS$, 
where $\cS$ denotes the $\infty$-category 
of spaces (Definition 1.2.16.1 of \cite{HTT}), alternatively 
we may replace such a map with the 
left fibration over $B^{\op}\times A$ 
that it classifies.  

The notion of correspondence has a 
straightforward interpretation at this level also: 
if $A$ and $B$ are simplicial sets we shall say that a simplicial 
map $p\colon X\to \Delta^1$ is  
a correspondence from $A$ to $B$ if 
there are isomorphisms $p^{-1}(0) \simeq B$ 
and $p^{-1}(1) \simeq A$ 
(see Definition~\ref{def:correspondence}).  The 
correspondences from $A$ to $B$ form the 
objects of a category $\Corr(A,B)$, which is a 
certain subcategory of the category $(\sSet)_{/\Delta^1}$ 
of simplicial sets over $\Delta^1$ 
(see Remark~\ref{rem:categ of correspondences}).  Correspondences 
of simplicial sets feature prominently 
in Lurie's discussion of adjoint functors 
in \cite{HTT}; they also play a 
role in \cite{BSP}. 

The notion of two-sided discrete fibration also 
extends to the context of simplicial 
sets.  In \cite{HTT}, Lurie introduced 
the notion of a {\em bifibration} $(f,g)\colon X\to A\times B$ 
which is an inner fibration together with a 
condition which encodes the idea that the 
fibers $X(a,b)$ of the map $(f,g)$ depend 
covariantly on $a$ and contravariantly on 
$b$ (this notion is also considered in 
\cite{J1}).  Bifibrations are the analog 
for simplicial sets of the notion of 
two-sided discrete fibration in category theory.  
For brevity we shall use Lurie's terminology of `bifibration' 
rather than `two-sided discrete fibration'.   
In the paper \cite{RV1} Riehl and Verity 
refer to bifibrations as {\em modules}; 
they play a key role in their study of the 
formal category theory of $\infty$-categories.  

One of our aims in this paper is to 
exhibit the bifibrations in simplicial 
sets from $A$ to $B$ as the fibrant objects 
of a model category.  In Section~\ref{subsec:biv model str} 
we shall prove the following result (see Theorem~\ref{thm:biv model str}): 

\begin{thmx} 
Let $A$ and $B$ be simplicial sets.  There is 
the structure of a left proper, combinatorial 
model category on $(\sSet)_{/(A\times B)}$ for 
which the cofibrations are the monomorphisms 
and the fibrant objects are the bifibrations 
from $A$ to $B$.  
\end{thmx} 

The existence of this model structure 
was known to Joyal (see \cite{J1}) but a construction 
of it has not appeared in the literature to date.  Following Joyal we call 
this model structure  
the {\em bivariant} model structure 
to reflect the covariant and contravariant 
nature of bifibrations. 

To establish the existence of this model structure 
we study bifibrations 
in some detail, replicating many properties 
of left and right fibrations established by 
Joyal and Lurie.  For instance we study the 
behaviour of bifibrations under 
exponentiation (Section~\ref{subsec:stability for 2-sided}), 
and we introduce the concept of {\em bivariant anodyne} 
map in $(\sSet)_{/(A\times B)}$ (see Section~\ref{subsec:biv anodyne}).  
We introduce the notion of {\em bivariant equivalence} 
(Section~\ref{sec:bivariant eqs}) 
and prove that a map $X\to Y$ between 
bifibrations in $(\sSet)_{/(A\times B)}$ 
is a bivariant equivalence if and only it 
is a fiberwise homotopy equivalence, generalizing 
the corresponding facts for left and right fibrations 
(Remark 2.2.3.3 of \cite{HTT}).  We also 
prove that a bifibration $X\to A\times B$ 
is a trivial Kan fibration if and only if 
its fibers are contractible Kan complexes.  
Again, this is a generalization of the corresponding 
facts for left and right fibrations (see Lemma 2.1.3.4 
of \cite{HTT}).  

In addition to the model structure for bifibrations, 
we also construct a model structure for correspondences.  
In Section~\ref{subsec:model str on corr} we prove 
the following result (see Theorem~\ref{thm:model str for corr}): 

\begin{thmx} 
Let $A$ and $B$ be $\infty$-categories.  There is the structure 
of a left proper, combinatorial model category on $\Corr(A,B)$ 
for which the cofibrations are the monomorphisms and the fibrant 
objects are the correspondences $X\to \Delta^1$ in $\Corr(A,B)$ 
for which $X$ is an $\infty$-category.  
\end{thmx} 
 
The model structure 
for correspondences is left induced (in the sense 
of \cite{BHKKRS}) from the Joyal model structure 
on the slice category $(\sSet)_{/B\star A}$.  
Its existence is well-known to experts --- 
it is stated, but not proved, in \cite{J1} 
and it is alluded to in \cite{HTT} for instance. 

Our other objective in this paper is to describe 
a series of Quillen equivalences linking the 
covariant model structure on $(\sSet)_{/(B^{\op}\times A)}$, 
the correspondence model structure on $\Corr(A.B)$, 
and the bivariant model structure on $(\sSet)_{/(A\times B)}$, 
which generalize the equivalences~\eqref{eq:one} 
and~\eqref{eq:two}.  Such a description 
has recently been given by Ayala and Francis 
in \cite{AF} at the level of $\infty$-categories.  
We shall refine the equivalences between 
$\infty$-categories that are established in \cite{AF} 
to Quillen equivalences between the model categories 
above.  In fact, we shall describe some additional 
Quillen equivalences, one of which is of a rather 
surprising nature.  

The twisted arrow category construction 
(see Construction 5.2.1.1 of \cite{HA}) associates 
to a simplicial set $X$ a new simplicial set 
$\Tw(X)$, equipped with a canonical map $\Tw(X)\to X^{\op}\times 
X$ which is a left fibration if $X$ is an $\infty$-category.  
If $X$ is a correspondence from 
$A$ to $B$, then base change along the map $B^{\op}\times A\to 
X^{\op}\times X$ induced by the inclusions $A\subseteq X$ 
and $B\subseteq X$ induces a functor 
$a^*\colon \Corr(A,B)\to (\sSet)_{/(B^{\op}\times A)}$ 
which participates in a series of adjunctions 
\[
\begin{tikzcd} 
(\sSet)_{/(B^{\op}\times A)} \arrow[r,shift left=4,"{a_!}"] 
\arrow[r,shift right=4,"{a_*}"]& \Corr(A,B) \arrow[l,"{a^*}"']  \\ 
\end{tikzcd} 
\]
In Section~\ref{subsec:subdivision} we shall prove 
the following result (see Theorem~\ref{thm: (delta_!,d) is a Quillen equivalence} 
and Theorem~\ref{thm:a^* a_* quillen equiv}): 
\begin{thmx} 
\label{thm:C} 
Let $A$ and $B$ be $\infty$-categories.  Then the adjoint pairs 
\[
a_!\colon (\sSet)_{/(B^{\op}\times A)} \rightleftarrows 
\Corr(A,B)\colon a^* 
\]
and 
\[
a^*\colon \Corr(A,B)\rightleftarrows (\sSet)_{/(B^{\op}\times A)}\colon a_* 
\]
are both Quillen equivalences for the correspondence model structure 
and the covariant model structure on $(\sSet)_{/(B^{\op}\times A)}$.
\end{thmx} 
Of note is the fact that the functor $a^*$ 
appears as both a left and right Quillen equivalence.  
There is a similar series of adjunctions 
\[
\begin{tikzcd} 
(\sSet)_{/(A\times B)} \arrow[r,shift left=4,"{d_!}"] 
\arrow[r,shift right=4,"{d_*}"]& \Corr(A,B) \arrow[l,"{d^*}"']  \\ 
\end{tikzcd} 
\]  
connecting the categories $(\sSet)_{/(A\times B)}$ 
and the category $\Corr(A,B)$, which is described 
in terms of the edgewise subdivision functor $\sd_2$ from 
\cite{BHM}.  In Section~\ref{subsec:quillen equivs} 
we prove (see Theorem~\ref{thm:d^* d_* QE}) 

\begin{thmx} 
\label{thm:D}
Let $A$ and $B$ be $\infty$-categories.  Then the adjoint pair 
\[
d^*\colon \Corr(A,B)\rightleftarrows (\sSet)_{/(A\times B)} \colon d_* 
\] 
is a Quillen equivalence for the correspondence model 
structure on $\Corr(A,B)$ and the bivariant model 
structure on $(\sSet)_{/(A\times B)}$.
\end{thmx}   

The adjoint pair 
$(d_!,d^*)$ is not a Quillen pair for these model 
structures; there is however another Quillen equivalence 
relating these model categories (see Theorem~\ref{thm:AF thm}).

In summary then the contents of this paper are as follows.  
In Section~\ref{sec:cov model str} we review some 
facts about the covariant model structure and Joyal's notion 
of dominant map that we will need later in the paper.  
In Section~\ref{sec:correspondences} we describe the model 
structure for correspondences and prove the existence of 
the Quillen equivalences from Theorem~\ref{thm:C} 
above.  In Section~\ref{sec:two sided fibns} we study 
the notion of a bifibration in simplicial sets; 
we introduce the concept of a bivariant anodyne map and 
bivariant equivalence.  We describe the bivariant model 
structure on $(\sSet)_{/(A\times B)}$ and prove the 
existence of the Quillen equivalence from 
Theorem~\ref{thm:D} above.   

Finally, we point out that several of the results in this paper 
seem to be known to experts, but equally proofs of them 
are missing from the literature; in this paper we fill these 
gaps.

\noindent 
{\bf Notation}: for the most part we use the notation 
and terminology from Lurie's books \cite{HTT} and \cite{HA}, 
except where we have indicated.  Thus $\sSet$ denotes the 
category of simplicial sets, $\h(S)$ denotes the homotopy 
category of a simplicial set $S$, etc.  Following the convention 
in \cite{HA}, we will say that a {\em left cofinal} map 
of simplicial sets is what is called a {\em cofinal} map in \cite{HTT} and that 
a map of simplicial sets is {\em right cofinal} if and only if 
its opposite is left cofinal.

\section{The covariant model structure} 
\label{sec:cov model str}
Let $S$ be a simplicial set.  We recall some 
features of the covariant model structure 
on the category $(\sSet)_{/S}$ of simplicial 
sets over $S$ from \cite{J1} 
and \cite{HTT}.  

\begin{notation} 
\label{not:simp mapping spaces}
Recall that the category $(\sSet)_{/S}$ is 
canonically enriched over $\sSet$.  If $X\to S$ 
and $Y\to S$ are objects of $(\sSet)_{/S}$ 
then the simplicial mapping space $\map_{S}(X,Y)$ 
is the simplicial set defined by the pullback diagram 
\[
\begin{tikzcd} 
\map_{S}(X,Y) \arrow[r] \arrow[d] & Y^X \arrow[d] \\ 
\Delta^0 \arrow[r] & S^X 
\end{tikzcd} 
\]
where the lower horizontal map corresponds to the 
structure map $X\to S$.  
\end{notation} 

\subsection{Covariant equivalences} 
\label{subsec:cov equivs} 
Recall that a map $f\colon X\to Y$ in $(\sSet)_{/S}$ 
is said to be a {\em covariant equivalence} 
if the induced map 
\[
\map_{S}(Y,L)\to \map_{S}(X,L) 
\]
is a weak homotopy equivalence for every left 
fibration $L\to S$.  
The covariant equivalences are the 
weak equivalences for the {\em covariant} 
model structure on $(\sSet)_{/S}$ introduced 
by Joyal and Lurie.  

\begin{theorem}[Joyal/Lurie] 
There is the structure of a left proper, combinatorial 
model category on $(\sSet)_{/S}$ for which 
\begin{itemize} 
  \item the weak equivalences are the covariant equivalences; 
  \item the cofibrations are the monomorphisms; and 
  \item the fibrant objects are the left fibrations.  
\end{itemize} 
\end{theorem} 

Dually there is the {\em contravariant} model structure 
on $(\sSet)_{/S}$, described in terms of right fibrations 
on $S$.  

The following theorem from \cite{S1} gives a 
very useful criterion for recognizing covariant 
equivalences.  

\begin{theorem}[\cite{S1}] 
\label{thm:7 from covariant paper}
Let $S$ be a simplicial set and let 
$f\colon X\to Y$ be a map in $(\sSet)_{/S}$.  
The following statements are 
equivalent: 
\begin{enumerate} 
\item $f$ is a covariant equivalence; 
\item the induced map $X\times_S R\to Y\times_S R$ 
is a weak homotopy equivalence for every right fibration 
$R\to S$; 
\item for every vertex $s$, and for every factorization 
$\Delta^0\to Rs\to S$ of the map $s\colon \Delta^0\to S$ 
into a right anodyne map followed by a right fibration, 
the induced map $X\times_S Rs\to Y\times_S Rs$ is a 
weak homotopy equivalence.  
\end{enumerate} 
\end{theorem} 

\subsection{The right cancellation property}
\label{subsec:right cancell}
Recall that a class of 
monomorphisms $\mathcal{A}$ in a category 
$\mathcal{C}$ is said to satisfy the 
{\em right cancellation property} 
if the following condition is satisfied: 
if $u$ and $v$ are composable morphisms 
in $\mathcal{C}$ such that $u\in \mathcal{A}$ 
and $vu\in \mathcal{A}$, then 
$v\in \mathcal{A}$ also.  Left anodyne 
maps are an important example of a 
class of maps with this property.    

\begin{proposition}[Joyal] 
\label{prop:right cancellation for left anodynes}
The class of left anodyne maps 
satisfies the right 
cancellation property.  
\end{proposition}

The following result from \cite{S1} 
gives a useful criterion for detecting when 
a given class of monomorphisms in $\sSet$ 
satisfying the right cancellation property contains the class of left 
anodyne maps.   

\begin{proposition}[\cite{S1}]
\label{prop:contains left anodynes}
Let $\mathcal{A}$ be a saturated class of monomorphisms in $\sSet$ 
which satisfies the right cancellation property.  Then the following statements 
are equivalent: 
\begin{enumerate} 
\item $\mathcal{A}$ contains the class of left anodyne morphisms; 
\item $\mathcal{A}$ contains the initial vertex maps $\Delta^{\set{0}}\to \Delta^n$ for all $n\geq 1$; 
\item $\mathcal{A}$ contains the horn inclusions $h^0_n\colon \Lambda^n_0\subseteq \Delta^n$ for all $n\geq 1$.  
\end{enumerate} 
\end{proposition}

\begin{remark} 
A similar criterion, framed in terms of the 
{\em spine inclusions} $\Delta^{\set{0,1}}\cup 
\cdots \cup \Delta^{\set{n-1,n}}\subseteq \Delta^n$, 
appears in \cite{BG}.  
\end{remark} 

From \cite{S2} we have another very useful example 
of a class of monomorphisms satisfying the right 
cancellation property.  

\begin{proposition}[\cite{S2}] 
\label{prop:inner anodyne right cancel}
The class of inner anodyne maps in $\sSet$ 
satisfies the right cancellation property.  
\end{proposition} 

We will make use of this fact in the proof of 
Proposition~\ref{prop:delta! preserves left anodyne}.  

\subsection{Dominant maps} 
\label{subsec:dominant maps}
In this section we recall some facts about the 
notion of {\em dominant} maps of simplicial 
sets introduced by Joyal (we shall need some of the results from 
this section in the proof of Lemma~\ref{lem:diag biv anodyne}  
in Section~\ref{sec:bivariant eqs}). 

\begin{definition}[Joyal] 
\label{def:dominant map}
A map $u\colon A\to B$ in $\sSet$ is said to be 
{\em dominant} if the right derived functor 
\[
\mathbf{R}u^*\colon \Ho((\sSet)_{/B})\to \Ho((\sSet)_{/A}) 
\]
is fully faithful for the contravariant model 
structures on $(\sSet)_{/A}$ and $(\sSet)_{/B}$.  
\end{definition} 

\begin{remark} 
The notion of dominant map is also studied by 
Gaitsgory and Rozenblyum, who use the term 
{\em contractible} map instead of dominant map 
(see Section 2.3 of \cite{GR}).  
\end{remark} 

\begin{remark} 
\label{rem:dom maps closed under retracts}
It follows immediately from Definition~\ref{def:dominant map} 
that dominant maps are closed under retracts and invariant 
under categorical equivalences.  
\end{remark} 

The following result is due to Joyal; we give a proof 
since we have not been able to find one in the literature to date. 

\begin{lemma}[Joyal] 
\label{lem:dom maps closed under cobase change along lf}
If $u\colon A\to B$ is dominant and $R\to B$ is a right 
fibration then the induced map $R\times_{B}A\to R$ 
is dominant.  
\end{lemma} 

\begin{proof}
Suppose given a dominant map $u\colon A\to B$ and suppose 
that $p\colon R\to B$ is a right fibration.  Form the 
pullback diagram 
\[
\begin{tikzcd} 
A\times_B R \arrow[r,"v"] \arrow[d,"q"] & R \arrow[d,"p"] \\ 
A \arrow[r,"u"] & B 
\end{tikzcd} 
\]
We need to prove that the right derived functor 
\[
\mathbf{R}v^*\colon \Ho((\sSet)_{/R})) \to 
 \Ho((\sSet)_{/(A\times_{B}R)}
\]
is fully faithful, where $(\sSet)_{/(A\times_{B}R)}$ and 
$(\sSet)_{/R}$ are equipped with the contravariant 
model structures.  We will prove 
that the counit  
\[
\epsilon_v\colon \mathbf{L}v_! \mathbf{R}v^*\to \mathrm{id} 
\]
is an isomorphism.  Since $p\colon R\to B$ is a right 
fibration, the left derived functor 
\[
\mathbf{L}p_!\colon \Ho((\sSet)_{/R}) \to 
\Ho((\sSet)_{/B}) 
\]
is conservative (Corollary 10.15 of \cite{J1}), and hence 
it suffices to prove that the image $\mathbf{L}p_!\epsilon$ 
is an isomorphism in $\Ho((\sSet)_{/B})$.  We have a 
natural isomorphism $\mathbf{L}p_!\mathbf{L}v_!\simeq 
\mathbf{L}u_!\mathbf{L}q_!$.  A straightforward argument, 
using the fact that $p$ is a right fibration, shows that 
the canonical natural transformation 
\[
\mathbf{L}q_!\mathbf{R} v^*\to \mathbf{R}u^*\mathbf{L}p_! 
\]
is a natural isomorphism.  Therefore $\mathbf{L}p_!\epsilon$ 
is isomorphic to the natural transformation 
\[
\epsilon_u\mathbf{L}p_!\colon \mathbf{L}u_!\mathbf{R}u^*\mathbf{L}p_! \to \mathbf{L}p_! 
\]
which is itself a natural isomorphism since $u$ is dominant.  
\end{proof} 

We state the following result which appears 
in \cite{J1} and \cite{GR}.  We first need some 
notation.  

\begin{notation} 
If $f\colon b\to b'$ is 
an edge of a simplicial set $B$ then we will write 
$B_{b//b'}$ for the double slice $(B_{b/})_{/f}$.  
If $B$ is the nerve of a category, 
then the simplicial set $B_{b//b'}$ is the nerve of the 
category of factorizations of the arrow $f\colon b\to b'$.  
\end{notation} 

\begin{lemma}
\label{lem:GR lemma}
Suppose that $A$ and $B$ are $\infty$-categories.  A map 
$u\colon A\to B$ is dominant if and only if the $\infty$-category $A\times_{B}B_{b//b'}$ 
is weakly contractible for every edge $f\colon b\to b'$ in $B$.  
\end{lemma} 

The proof of this statement is reasonably straightforward 
and is left to the reader.  
We note the following consequences.    

\begin{remark} 
It follows easily that a map $u\colon A\to B$ 
of simplicial sets is dominant if and only if the 
opposite map $u^{\op}\colon A^{\op}\to B^{\op}$ is 
dominant (recall that dominant maps are invariant 
under categorical equivalences).  
\end{remark} 

\begin{remark} 
\label{rem:dom and cofinal}
It follows, using Theorem 4.1.3.1 from \cite{HTT}, that 
every dominant map is left cofinal and right cofinal.  
\end{remark} 

\begin{lemma} 
\label{lem:prods of dom maps}
If $u\colon A\to B$ and $v\colon C\to D$ are dominant 
maps of simplicial sets, then the product $u\times v\colon A\times C 
\to B\times D$ is also dominant.  
\end{lemma} 

\begin{proof}
It suffices to prove that if $u\colon A\to B$ is 
dominant then $u\times \mathrm{id}\colon A\times C 
\to B\times C$ is dominant for any simplicial set 
$C$.  Since dominant maps are invariant under 
categorical equivalence, we may suppose without 
loss of generality that $A$, $B$ and $C$ 
are $\infty$-categories.  This follows 
immediately from Lemma~\ref{lem:GR lemma}, 
using the fact that for any vertices 
$b\in B$ and $c\in C$ we have a pullback diagram 
\[
\begin{tikzcd} 
A\times_{B} B_{b/}\times C_{c/} 
\arrow[r] \arrow[d] & 
B_{b/}\times C_{c/} \arrow[d]      \\ 
A\times C \arrow[r] & B\times C 
\end{tikzcd} 
\]
involving the undercategories $B_{b/}$ 
and $C_{c/}$, and similarly for overcategories.      
\end{proof}

We conclude this section with the following  
useful example of a dominant map.  

\begin{lemma} 
\label{lem:diagonal dominant}
For every $n\geq 0$ the diagonal map $\Delta^n\to 
\Delta^n\times \Delta^n$ is dominant.  
\end{lemma} 

\begin{proof} 
The diagonal map $\Delta^n\to \Delta^n\times 
\Delta^n$ is a retract of the diagonal map 
$(\Delta^1)^n\to (\Delta^1)^n\times (\Delta^1)^n$.  
Therefore, since dominant maps are closed under retracts 
(Remark~\ref{rem:dom maps closed under retracts}) and 
products (Lemma~\ref{lem:prods of dom maps}) we are reduced to proving 
that $\Delta^1\to \Delta^1\times \Delta^1$ is dominant.  
This can be proven using Lemma~\ref{lem:GR lemma} and 
a case by case analysis.    
\end{proof} 

\subsection{Inner anodyne maps and inner fibrations} 
\label{subsec:inner}
We close this section by recording a couple of straightforward 
results about inner fibrations and inner anodyne maps that we will 
need later in the paper.  
\begin{lemma} 
\label{lem:inner implies kan}
Suppose that $p\colon S\to T$ is an inner fibration, where $S$ and $T$ are Kan complexes.  
If $p$ has the right lifting property against the map $\Delta^{\set{0}}\to \Delta^1$ then $p$ is a Kan fibration.  
\end{lemma} 

\begin{proof}  
It suffices to prove that $p$ is a left fibration, since $T$ is a Kan complex.  Every edge of 
$S$ is an equivalence and hence is $p$-cocartesian (Proposition 2.4.1.5 of \cite{HTT}).  
Therefore $p$ has the right lifting property against every horn inclusion of the form 
$\Lambda^n_0\subseteq \Delta^n$, $n\geq 2$ (Remark 2.4.1.4 of \cite{HTT}). Therefore, 
invoking the assumption that $p$ has the right lifting 
property against the map $\Delta^{\set{0}}\to \Delta^1$, it follows that 
$p$ is a left 
fibration.  
\end{proof}

\begin{lemma} 
\label{lem:inner anodyne lemma}
Let $B$ be an $\infty$-category.  Suppose that $i\colon A\to B$ is an 
acyclic cofibration in the Joyal model structure on 
$\sSet$, such that $i$ is a bijection on 0-simplices.  Then 
$i$ is inner anodyne.  
\end{lemma} 

\begin{proof} 
Factor $i$ as $i = pj$, where $j\colon A\to B'$ is inner anodyne and $p\colon B'\to B$ 
is an inner fibration.  Then $p$ is a categorical fibration, since $p$ is bijective on 
objects and $B$ is an $\infty$-category.  Therefore $p$ is a trivial Kan fibration and 
hence has a section $s\colon B\to B'$, which exhibits $i$ as a retract of $j$.  Hence 
$i$ is inner anodyne.  
\end{proof}

\section{Correspondences} 
\label{sec:correspondences}
\subsection{The category of correspondences from $A$ to $B$} 
\label{subsec:category of correspondences}
We recall the notion of a correspondence between simplicial sets from 
Section 2.3.1 and Section 5.2.1 of \cite{HTT}.  

\begin{definition}[Lurie] 
\label{def:correspondence}
Let $A$ and $B$ be simplicial sets.    
A {\em correspondence} from 
$A$ to $B$ is a map $p\colon X\to \Delta^1$ with $p^{-1}(0) = B$ and 
$p^{-1}(1) = A$.
\end{definition} 

\begin{remark} 
We do not require that the map $p$ in the above definition 
is an inner fibration; we will reserve the term {\em fibrant} 
correspondence to describe such a map (see 
Section~\ref{subsec:model str on corr} below).  
Note also that we call a correspondence from $A$ to 
$B$ is what is called a correspondence from $B$ to 
$A$ in \cite{HTT}.  
\end{remark} 

\begin{remark} 
\label{rem:categ of correspondences}
We write $\Corr(A,B)$ for the subcategory of 
$(\sSet)_{/\Delta^1}$ whose objects are the correspondences from $A$ to $B$ and 
where a map $f\colon X\to Y$ is a map in $(\sSet)_{/\Delta^1}$ such that 
$f|A = \mathrm{id}_A$ and $f|B = \mathrm{id}_B$.  
\end{remark}  
  
\begin{remark} 
Clearly $B\sqcup A$, equipped with the canonical map $B\sqcup A\to \partial \Delta^1\to \Delta^1$ 
is an initial object of $\Corr(A,B)$.  If $p\colon X\to \Delta^1$ is a correspondence in 
$\Corr(A,B)$ and $u\colon \Delta^n\to X$ is a simplex, then the composite map 
$pu\colon \Delta^n\to \Delta^1$ has a unique decomposition $pu = i\star f$, where 
$i\colon \Delta^k\to \Delta^0$ and $f\colon \Delta^{n-k-1}\to \Delta^0$.  It follows that 
$ui$ factors through $B$, and $uf$ factors through $A$.  Therefore $ui\star uf$ is an $n$-simplex 
of $B\star A$.  This defines a unique map $X\to B\star A$, 
from which it follows that $B\star A$ is a terminal object of 
$\Corr(A,B)$.  
\end{remark} 

\begin{remark} 
\label{rem:reflector L}
There is a canonical full inclusion $i\colon \Corr(A,B)\hookrightarrow (\sSet)_{/B\star A}$.  
The inclusion $i$ has a left adjoint $L\colon (\sSet)_{/B\star A}\to \Corr(A,B)$ which exhibits 
$\Corr(A,B)$ as a full reflective subcategory of $(\sSet)_{/B\star A}$.  The reflector $L$ is 
defined on objects as follows: if $X\in (\sSet)_{/B\star A}$ with structure map $p$ then $L(X)$ is the 
correspondence defined by the pushout diagram 
\[
\begin{tikzcd}
p^{-1}(B\sqcup A) \arrow[d] \arrow[r] & X \arrow[d] \\ 
B\sqcup A \arrow[r] & L(X) 
\end{tikzcd}
\]  
\end{remark} 

\begin{remark} 
\label{rem:presentable}
As a full reflective subcategory of the presentable category 
$(\sSet)_{/B\star A} $, it follows (see 
Corollary 6.24 of \cite{AR}) that $\Corr(A,B)$ is presentable 
(here presentable is understood in the sense of Definition A.1.1.2 of 
\cite{HTT}).  In particular it follows that $\Corr(A,B)$ has all 
limits and colimits.  
\end{remark}

\subsection{The model structure on correspondences} 
\label{subsec:model str on corr}
Suppose now that $A$ and $B$ are $\infty$-categories.  The Joyal model structure 
on $\sSet$ induces a model structure on the slice category $(\sSet)_{/B\star A} $ 
in the usual way.   By definition, a map $X\to B\star A$ is a fibrant object in 
this model structure if and only if it is a categorical fibration.  

Let us say that an object in $\Corr(A,B)$ is {\em fibrant} if and only if the canonical map 
$X\to B\star A$ is fibrant in the induced model structure on $(\sSet)_{/B\star A} $.  
We have the following result. 

\begin{lemma} 
\label{lem:char of fib in corr}
Let $X\in \Corr(A,B)$, where $A$ and $B$ are 
$\infty$-categories.  The following statements are equivalent: 
\begin{enumerate} 
\item $X$ is fibrant 
\item the canonical map 
$p\colon X\to B\star A$ is an inner fibration
\item the canonical map $X\to \Delta^1$ is an inner fibration 
\item $X$ is an $\infty$-category.
\end{enumerate}  
\end{lemma} 

\begin{proof} 
The equivalence of statements (3) and (4) is clear.  It is also clear that 
(1) $\Rightarrow$ (2).  We prove that (2) $\Rightarrow$ (1).  Suppose 
that the canonical map $p\colon X\to B\star A$ is an inner fibration.     
Since $B\star A$ is an $\infty$-category (Proposition 1.2.8.3 of \cite{HTT}), 
it follows that $X$ is an $\infty$-category. Hence $p\colon X\to B\star A$ 
is a categorical fibration if and only if $\h(X)\to \h(B\star A)$ is an 
isofibration, i.e.\ has the right lifting property against the inclusion $\set{0}\to J$, 
where $J$ denotes the groupoid interval.  
We have $\h(B\star A) = \h(B)\star \h(A)$, where the right hand side 
denotes the join of the categories $\h(B)$ and $\h(A)$.  Therefore the only 
isomorphisms in $\h(B\star A)$ are represented by equivalences in $B$ or equivalences 
in $A$. Since $X\in \Corr(A,B)$ these equivalences lift automatically to equivalences  
in $X$.  Finally, to complete the proof, we shall prove that (2) $\iff$ (4).  
The implication (2) $\Rightarrow$ (4) is immediate from the fact that 
$B\star A$ is an $\infty$-category.  To prove the converse, assume that 
$X$ is an $\infty$-category and consider a commutative diagram 
\[
\begin{tikzcd} 
\Lambda^n_i \arrow[d] \arrow[r,"u"] & X \arrow[d,"p"] \\ 
\Delta^n \arrow[r,"v"'] & B\star A 
\end{tikzcd} 
\]
We will prove that this map is compatible 
with the projection to $B\star A$.  The map $v\colon \Delta^n\to B\star A$ decomposes 
as $v = x\star y$, where $x\colon \Delta^k\to B$ and $y\colon \Delta^{n-1-k}\to A$.  
If $k=-1$ or $k=n$ then we can find a diagonal filler 
for the diagram above since both $A$ and $B$ are 
$\infty$-categories.   Otherwise, we have $\Delta^k\subseteq \Lambda^n_i$ 
and $\Delta^{n-1-k}\subseteq \Lambda^n_i$.  
Since $X$ is an $\infty$-category, we may extend the map 
$u$ along the inner horn inclusion $\Lambda^n_i\hookrightarrow\Delta^n$  
to obtain a map $w\colon \Delta^n\to X$.   It follows that 
$w|\Delta^k = x$ and $w|\Delta^{n-1-k}=y$ and hence $pw = v$.    
\end{proof}  

More generally, we have  

\begin{lemma} 
\label{lem:char of fibns between fib corr}
Suppose that $f\colon X\to Y$ is a map between fibrant objects in $\Corr(A,B)$.  
Then the underlying map of simplicial sets is a categorical fibration if 
and only if it is an inner fibration.  
\end{lemma} 

\begin{proof} 
As in the proof of Lemma~\ref{lem:char of fib in corr} above, we need to check that 
the induced map $\h(X)\to \h(Y)$ is an isofibration of categories.  
An isomorphism in $\h(Y)$ maps to an isomorphism in $\h(B\star A)$ and 
hence is represented by either an equivalence in $B$ or an equivalence in $A$.  
The result then follows since $f$ is a map in $\Corr(A,B)$.  
\end{proof}

\begin{theorem} 
\label{thm:model str for corr}
Let $A$ and $B$ be $\infty$-categories.  There exists the structure of a left proper, 
combinatorial model category on $\Corr(A,B)$ for which a map $X\to Y$ is a 
\begin{itemize} 
	\item cofibration if the underlying map of simplicial sets is a monomorphism; 
	\item weak equivalence if the underlying map of simplicial sets is a categorical 
	equivalence
\end{itemize} 
and for which the fibrant objects are the correspondences $X$ whose 
underlying simplicial set is an $\infty$-category.   
\end{theorem} 

\begin{proof} 
We use Proposition A.2.6.13 from \cite{HTT}.  To begin with, as observed  
in Remark~\ref{rem:presentable} above, 
$\Corr(A,B)$ is presentable.  We verify the three conditions (1), (2) 
and (3) from op.\ cit. Let $C$ denote the class of cofibrations in $\Corr(A,B)$ 
and let $W$ denote the class of weak equivalences in $\Corr(A,B)$.  The weakly 
saturated class of monomorphisms in $(\sSet)_{/B\star A} $ is generated by the 
set of boundary inclusions $\partial\Delta^n\subseteq \Delta^n$ in $(\sSet)_{/B\star A} $ 
for $n\geq 0$.  The simplices in $B\star A$ are of the following three types: 
$x\star \emptyset\colon \Delta^n\star \emptyset\to B\star A$, $x\star y\colon 
\Delta^m\star \Delta^n\to B\star A$, and $\emptyset\star y\colon \emptyset\star 
\Delta^n\to B\star A$.  It follows that $C$ is generated as a weakly saturated 
class by the set $C_0$ of monomorphisms in $\Corr(A,B)$ of the form 
\begin{align*}
& (\emptyset\star \partial\Delta^n)\cup_{\emptyset\sqcup \partial\Delta^n}(B\sqcup A)\to 
(\emptyset\star \Delta^n)\cup_{\emptyset\sqcup \Delta^n}(B\sqcup A)                       \\ 
& (\partial\Delta^m\star \emptyset)\cup_{\partial\Delta^m\sqcup \emptyset}(B\sqcup A) \to 
(\Delta^m\star \emptyset)\cup_{\Delta^m\sqcup\emptyset}(B\sqcup A)                       \\ 
& (\partial\Delta^m\star \Delta^n\cup \Delta^m\star \partial\Delta^n)_{
(\partial\Delta^m\star \Delta^n\sqcup \Delta^m\star \partial\Delta^n)}(B\sqcup A)\to 
(\Delta^m\star\Delta^n)\cup_{\Delta^m\sqcup \Delta^n}(B\sqcup A)
\end{align*}

For (1), observe that the class $W$ is the 
inverse image of the class of categorical equivalences of simplicial sets 
under the forgetful functor $\Corr(A,B)\to \sSet$.  It follows that $W$ is perfect 
by Corollary A.2.6.12 of \cite{HTT}. 

(2) follows immediately from the fact that the Joyal model structure on $\sSet$ 
is left proper.  

For (3), observe that if $f\colon X\to Y$ is a map in $\Corr(A,B)$ which has the 
right lifting property with respect to every morphism in $C_0$, then $f$ is a trivial 
Kan fibration.  For then $f$ has the right lifting property with respect to every 
monomorphism in $\sSet$.  

For the characterization of the fibrant objects, observe that $X\to B\star A$ 
has the right lifting property with respect to all maps in $C\cap W$ if and only 
if the underlying map of simplicial sets is a categorical fibration.  We then 
apply Lemma~\ref{lem:char of fib in corr}.    
\end{proof} 

\begin{remark} 
The model structure for correspondences is the left 
induced model structure (in the sense of \cite{BHKKRS}) 
on $\Corr(A,B)$ associated to the 
adjoint pair $(L,i)$ and the Joyal model 
structure on $(\sSet)_{/B\star A}$.  
\end{remark} 

\begin{remark} 
\label{rem:simp enrichment}
The category $(\sSet)_{/B\star A} $ has a natural 
structure as a simplicial category 
which is tensored and cotensored over 
$\sSet$.  This structure induces on $\Corr(A,B)$ 
the structure of a simplicial category which is 
tensored and cotensored over $\sSet$.  If $X\in \Corr(A,B)$ and $K$ is a simplicial 
set, then the cotensor $X\otimes K$ is defined by the pushout diagram 
\[
\begin{tikzcd} 
(B\times K)\sqcup (A\times K) \arrow[r] \arrow[d] & X\times K \arrow[d] \\ 
B\sqcup A \arrow[r] & X\otimes K 
\end{tikzcd} 
\]
where the left hand vertical map is induced by the canonical projections 
$B\times K\to B$ and $A\times K\to A$.  The construction $X\otimes K$ extends 
to define a functor $\otimes \colon \Corr(A,B)\times \sSet\to \Corr(A,B)$.  Observe 
that for a fixed correspondence $X\in \Corr(A,B)$, the induced functor 
$X\otimes (-)\colon \sSet\to \Corr(A,B)$ commutes with colimits.  

If $X$ is again a correspondence in $\Corr(A,B)$ and $K$ is a simplicial set, 
then the cotensor $^KY$ is defined by the pullback diagram 
\[
\begin{tikzcd} 
^KY \arrow[d] \arrow[r] & Y^K \arrow[d] \\ 
B\star A \arrow[r] & (B\star A)^K 
\end{tikzcd} 
\]
where the lower horizontal map is conjugate to the canonical projection 
$(B\star A)\times K\to B\star A$.  Note that $^KY$ so defined is a correspondence: 
the two squares in the commutative diagram 
\[
\begin{tikzcd} 
B\sqcup A \arrow[d] \arrow[r] & B^K\sqcup A^K \arrow[d] \arrow[r] & Y^K \arrow[d] \\ 
B\sqcup A \arrow[r] & B^K\sqcup A^K \arrow[r] & (B\star A)^K 
\end{tikzcd} 
\]
are both pullbacks, and the composite map $B\sqcup A\to (B\star A)^K$ factors 
as $B\sqcup A\to B\star A\to (B\star A)^K$, where the second map is the canonical map 
above.  The adjointness $(-)\otimes K\dashv\  ^K(-)$ is clear.
It follows 
(Lemma II 2.2 of \cite{GJ}) that $\Corr(A,B)$ has the structure of a simplicial category, 
tensored and cotensored over $\sSet$.      
\end{remark} 

\begin{remark} 
Let $A$ and $B$ be $\infty$-categories.  The model structure for correspondences on 
$\Corr(A,B)$ is enriched over the Joyal model structure via the simplicial enrichment 
described in Remark~\ref{rem:simp enrichment}.  
\end{remark} 

\subsection{Distributors to correspondences, and back again} 
\label{subsec:subdivision}
Recall that the 
{\em edgewise subdivision} of a simplicial set $X$ (in the sense of 
Segal \cite{Segal}) is defined by composing the functor $X\colon \Delta^{\op}\to \Sets$
with the opposite of the `doubling functor' 
\begin{align*}
& d\colon \Delta\to \Delta \\ 
& [n]\mapsto [n]^{\op}\star [n] 
\end{align*} 
This construction can be used to relate the category 
$(\sSet)_{/(B^{\op}\times A)}$ with the 
category of correspondences from $A$ to $B$.  The relation 
is as follows.    

Observe that the doubling functor above induces a functor between simplex categories   
\[
\sigma\colon \Delta_{/(B^{\op}\times A)}\to \Delta_{/(B\star A)}  
\]
which sends a pair $(u,v)\colon \Delta^n\to B^{\op}\times A$ to  
\[
\sigma(u,v) = u^{\op}\star v\colon (\Delta^n)^{\op}\star \Delta^n\to B\star A. 
\]
The functor $\sigma$ induces an adjunction 
\[
\sigma_!\colon (\sSet)_{/(B^{\op}\times A)}\rightleftarrows (\sSet)_{/B\star A} \colon \sigma^* 
\]
and in fact the functor $\sigma^*$ has a further right adjoint $\sigma_*\colon (\sSet)_{/(B^{\op}\times A)}
\to (\sSet)_{/B\star A} $.

We make the following observations about the functors $\sigma_!$ and $\sigma^*$.  

\begin{lemma} 
\label{lem:sigma_! mono}
The functor $\sigma_!\colon (\sSet)_{/(B^{\op}\times A)}\to (\sSet)_{/B\star A}$ sends 
monomorphisms to monomorphisms.  
\end{lemma} 

\begin{proof}
Let $f\colon X\to Y$ be a monomorphism in $(\sSet)_{/(B^{\op}\times A)}$.  The map 
$\sigma_!(f)\colon \sigma_!(X)\to \sigma_!(Y)$ is a monomorphism if and only 
if the underlying map of simplicial sets is a monomorphism; therefore it suffices 
to check that for every $n\geq 0$ the induced map $\sigma_!(f)_n\colon \sigma_!(X)_n\to 
\sigma_!(Y)_n$ is a monomorphism of sets.  But $\sigma_!(f)_n$ is 
easily seen to be the map 
$f_{2n+1}\colon X_{2n+1}\to Y_{2n+1}$ which is a monomorphism by hypothesis.    
\end{proof} 

\begin{lemma} 
\label{lem:sigma*}
If $f\colon C\to A$ and $g\colon D\to B$ are maps determining objects $D\star C$ and $D^{\op}\times C$ 
of $(\sSet)_{/B\star A}$ and $(\sSet)_{/(B^{\op}\times A)}$  
respectively, then 
\[
\sigma^*(D\star C) = D^{\op}\times C.  
\]
\end{lemma} 

\begin{proof} 
If $\phi\colon (\Delta^n)^{\op}\star \Delta^n\to D\star C$ is a map in $(\sSet)_{/B\star A}$ such 
that $(g\star f)\phi = u^{\op}\star v$ where $u\colon \Delta^n\to B^{\op}$ and $v\colon 
\Delta^n\to A$, then 
$\phi$ is necessarily of the form $\phi = x^{\op}\star y$ for unique simplices $x\colon \Delta^n\to D^{\op}$ 
and $y\colon \Delta^n\to C$ with $x^{\op} = \phi|(\Delta^n)^{\op}$ and $y = \phi|\Delta^n$.  The key 
point here is the canonical functor $\star \colon \sSet\times \sSet\to (\sSet)_{/\Delta^1}$, defined by 
the join operation, is fully faithful (see Proposition 3.5 in \cite{J1}).    
\end{proof}

We write $a_!\colon (\sSet)_{/(B^{\op}\times A)}\to \Corr(A,B)$ for the composite functor 
$a_!:= L\sigma_!$, and we write $a^*\colon \Corr(A,B)\to (\sSet)_{/(B^{\op}\times A)}$ for 
the composite functor $a^*:= \sigma^* i$.  The functors $a_!$ and $a^*$ form 
an adjoint pair $(a_!,a^*)$.  

\begin{remark} 
\label{rem:calculation of unit map}
Observe that if $f\colon C\to A$ and $g\colon D\to B$ are maps then the functor $\sigma^*\colon (\sSet)_{/B\star A}\to 
(\sSet)_{/(B^{\op}\times A)}$ sends the object $D\sqcup C$ to the initial object $\emptyset$ of $(\sSet)_{/(B^{\op}\times A)}$. 
It follows that the endo-functor $a^*a_!$ of $(\sSet)_{/(B^{\op}\times A)}$ is isomorphic to the endo-functor $\sigma^*\sigma_!$.  
It follows that $a^*a_!$ preserves all colimits and hence is determined by its value on the 
$n$-simplices $\Delta^n\to B^{\op}\times A$ for $n\geq 0$.  A short 
calculation, using Lemma~\ref{lem:sigma*}, shows that in fact the unit map 
$\Delta^n\to a^*a_!(\Delta^n)$ is isomorphic to the diagonal map $\Delta^n\to \Delta^n\times 
\Delta^n$, where $\Delta^n\times \Delta^n$ is regarded as an object of 
$(\sSet)_{/(B^{\op}\times A)}$ via the map $f\times g\colon \Delta^n\times \Delta^n\to B^{\op}\times A$, where 
$(f,g)\colon \Delta^n\to B^{\op}\times A$.    
\end{remark}

\begin{remark} 
\label{rem:a* right adjoint}
We observe that the functor $a^*\colon \Corr(A,B)\to (\sSet)_{/(B^{\op}\times A)}$ 
has a right adjoint $a_*\colon (\sSet)_{/(B^{\op}\times A)}\to \Corr(A,B)$.  To see this, 
it suffices to prove that the functor $a^*\colon \Corr(A,B)\to (\sSet)_{/(B^{\op}\times A)}$ preserves 
colimits.  This follows from the fact that $\sigma^*$ preserves colimits and the 
fact that $\sigma^*iL = \sigma^*$ (this last fact can easily be seen using the 
observation made in Remark~\ref{rem:calculation of unit map}).  
If $X\to B^{\op}\times A$ is an object of $(\sSet)_{/(B^{\op}\times A)}$ 
then $a_*(X)$ is the correspondence from $A$ to $B$ such that 
\[
(\sSet)_{/B\star A}(\Delta^n,ia_*(X)) = 
(\sSet)_{/(B^{\op}\times A)}(a^*L(\Delta^n),X)
\] 
for every $n$-simplex 
$\Delta^n\to B\star A$.  
\end{remark} 

\begin{proposition} 
\label{prop:delta! preserves left anodyne}
The functor $a_!\colon (\sSet)_{/(B^{\op}\times A)}\to 
\Corr(A,B)$ sends left anodyne morphisms in $(\sSet)_{/(B^{\op}\times A)}$ 
to inner anodyne morphisms in $\Corr(A,B)$.    
\end{proposition}

\begin{proof} 
From Lemma~\ref{lem:sigma_! mono} we have that $a_!$ sends monomorphisms to monomorphisms.
Let $\mathcal{A}$ denote the class of all monomorphisms $v$ in 
$(\sSet)_{/B^{\op}\times A}$ such that the underlying map of simplicial 
sets $a_!(v)$ is inner anodyne.  We need 
to prove that every left anodyne morphism in $(\sSet)_{/(B^{\op}\times A)}$ is contained in 
$\mathcal{A}$.  Therefore, by Proposition~\ref{prop:contains left anodynes} 
it is sufficient to prove that $\mathcal{A}$ is saturated, satisfies the right 
cancellation property, and that the initial vertex maps $i_n\colon \Delta^{0}\to \Delta^n$ 
are contained in $\mathcal{A}$ for all $n\geq 0$.  By Proposition~\ref{prop:inner anodyne right cancel},  
the class of inner anodyne maps in $\sSet$ has the right cancellation property; 
the functoriality of $a_!$ then implies that $\mathcal{A}$ also has the right cancellation property.   
Likewise it is clear that $\mathcal{A}$ is a saturated class of monomorphisms 
since the inner anodyne maps in $\sSet$ form a saturated class 
and $a_!$ is a left adjoint.    

Let $n\geq 0$; we show that $i_n\colon \Delta^0\to \Delta^n$ is contained 
in $\mathcal{A}$.  The map $a_!(i_n)$ is a pushout of the map 
\[
(\Delta^0)^{\op}\star \Delta^0\cup_{(\Delta^0)^{\op}\sqcup \Delta^0}((\Delta^n)^{\op}\sqcup \Delta^n)
\to (\Delta^n)^{\op}\star \Delta^n, 
\]
and hence it suffices to prove that this last map is inner anodyne.  This map factors as 
\[
(\Delta^0)^{\op}\star \Delta^0\cup (\Delta^n)^{\op}\cup 
\Delta^n \to (\Delta^n)^{\op}\star \Delta^0\cup\Delta^n
\to (\Delta^n)^{\op}\star \Delta^n.  
\]
The first map in this composite is a pushout of the map 
$(\Delta^0)^{\op}\star \Delta^0\cup (\Delta^n)^{\op}\to (\Delta^n)^{\op}\star \Delta^0$ 
which is inner anodyne by Lemma 2.1.2.3 of \cite{HTT}, and the second map in 
this composite is inner anodyne by another application of this lemma.  
Hence 
$\mathcal{A}$ contains the left anodyne morphisms in $\sSet$, which completes the proof of the 
proposition.  
\end{proof}

The following corollary is straightforward.  

\begin{corollary} 
\label{corr:d preserves left fibns}
The functor $a^*\colon \Corr(A,B)\to (\sSet)_{/(B^{\op}\times A)}$ sends 
inner fibrations in $\Corr(A,B)$ to left fibrations in $(\sSet)_{/(B^{\op}\times A)}$.  
\end{corollary} 

\begin{remark} 
\label{rem:twisted arrow category} 
If $A$ is an $\infty$-category then the image of the functor $a^*$ on the correspondence 
$A\times I$ in $\Corr(A,A)$ is precisely the canonical map $\Tw(A)\to A^{\op}\times A$, 
where $\Tw(A)$ denotes the {\em twisted arrow category} of $A$ (see Construction 5.2.1.1.\  of \cite{HA}; 
note also that $\Tw(A)$ is precisely the Segal edge-wise subdivision of $A$ 
from \cite{Segal}).  Thus  
Corollary~\ref{corr:d preserves left fibns} gives an 
alternative proof that this canonical map is a left fibration (we hasten 
to point out that this proof proceeds along similar lines to 
the proof of Proposition 1.1 in \cite{BG}).  
\end{remark}

\begin{remark} 
\label{rem:mapping spaces}
Suppose that $X$ is an $\infty$-category and that $x$ is an object of $X$.  Observe that 
$\Tw(X)|\set{x}\times X$ may be described as the diagonal of the bisimplicial set 
$X_{\Delta^{\op}/}\colon \Delta^{\op}\to \sSet$ defined by 
\[
[n]\mapsto X_{(\Delta^n)^{\op}/} 
\]
where the slice $X_{(\Delta^n)^{\op}/}$ is defined by the map $(\Delta^n)^{\op}\to X$ 
given as the composite 
\[
(\Delta^n)^{\op}\to (\Delta^0)^{\op}\xrightarrow{x} X.  
\]
The bisimplicial set $X_{\Delta^{\op}/}$ 
has a canonical augmentation over the constant bisimplicial set $X_{x/}$ 
which in degree $n$ is the canonical map 
$X_{/(\Delta^n)^{\op}}\to X_{x/}$ induced 
by the right anodyne map $(\Delta^n)^{\op}\to 
(\Delta^{\set{0}})^{\op}$.  Applying the 
diagonal functor $d$ gives a categorical equivalence $\Tw(X)|\set{x}\times X\to X_{x/}$ 
forming part of a commutative diagram 
\[
\begin{tikzcd}[column sep=small] 
\Tw(X)|\set{x}\times X\arrow[rr] \arrow[dr] & & X_{x/} \arrow[dl] \\ 
& X 
\end{tikzcd} 
\]
Since the maps $\Tw(X)|\set{x}\times X\to X$ and $X_{x/}\to X$ are left 
fibrations, it follows that the map $\Tw(X)|\set{x}\times X\to X_{x/}$ is a 
fiberwise homotopy equivalence.  In particular it follows that there is a homotopy 
equivalence between the fiber $\Tw(X)(x,y)$ and $\Hom^L_X(x,y)$ for all objects 
$x$ and $y$ of $X$.    
\end{remark}

\begin{remark} 
\label{rem:alt desc of a*}
If $X\in \Corr(A,B)$ is a correspondence, then $a^*X\to B^{\op}\times A$ is 
the left hand vertical map in the pullback diagram 
\[
\begin{tikzcd} 
a^*X \arrow[r] \arrow[d] & \Tw(X) \arrow[d] \\ 
B^{\op}\times A \arrow[r] & X^{\op}\times X 
\end{tikzcd} 
\]
where the lower horizontal map is induced by the inclusions $A\subseteq X$ and 
$B\subseteq X$.  
\end{remark} 

\begin{proposition} 
Let $A$ and $B$ be $\infty$-categories.  The adjunction 
\[
a_!\colon (\sSet)_{/(B^{\op}\times A)}\rightleftarrows \Corr(A,B)\colon a^*  
\]
is a Quillen adjunction 
for the covariant model structure on $(\sSet)_{/(B^{\op}\times A)}$ 
and the model structure for correspondences on $\Corr(A,B)$.  
\end{proposition} 

\begin{proof}
The functor $a_!$ sends monomorphisms to monomorphisms, and hence $a^*$ sends trivial 
fibrations to trivial fibrations.  We prove that $a^*$ sends fibrations between fibrant 
objects in $\Corr(A,B)$ to covariant fibrations in $(\sSet)_{/(B^{\op}\times A)}$.  
By Corollary~\ref{corr:d preserves left fibns} the functor $a^*$ sends inner fibrations 
in $\Corr(A,B)$ to left fibrations in $(\sSet)_{/(B^{\op}\times A)}$.  It follows that 
$a^*$ sends fibrations between fibrant 
objects in $\Corr(A,B)$ to covariant fibrations in $(\sSet)_{/(B^{\op}\times A)}$.  
This completes the proof of the proposition.     
\end{proof} 

In \cite{AF} Ayala and Francis prove that there is a 
categorical equivalence between the $\infty$-category 
$\Fun(B^{\op}\times A,\cS)$ and an $\infty$-category of 
correspondences from $A$ to $B$.  The following theorem 
refines their result to a statement at the level of 
model categories (this latter statement is also certainly 
well-known; it is stated without proof in 
\cite{J1} and it is also stated as Remark 2.3.1.4 in \cite{HTT}).  We shall 
give a proof, since one has not appeared in the literature so far, and 
since we shall need some results obtained in the course of the proof 
for the proof of Theorem~\ref{thm:a^* a_* quillen equiv}.  
  
\begin{theorem} 
\label{thm: (delta_!,d) is a Quillen equivalence}
Let $A$ and $B$ be $\infty$-categories.  Then the Quillen adjunction 
\[
a_!\colon (\sSet)_{/(B^{\op}\times A)}\rightleftarrows \Corr(A,B)\colon a^*.  
\]
extends to a Quillen equivalence 
for the covariant model structure on $(\sSet)_{/(B^{\op}\times A)}$ 
and the model structure for correspondences on $\Corr(A,B)$.  
\end{theorem}

\begin{proof}
We prove that (i) $a^*$ reflects weak equivalences between fibrant objects, and (ii) 
if $X\in (\sSet)_{/(B^{\op}\times A)}$, then $X\to a^*\boldR a_!X$ is a 
covariant equivalence in $(\sSet)_{/(B^{\op}\times A)}$, 
where $\boldR a_!X$ denotes a fibrant replacement of $a_!X$ in the 
model structure for correspondences on $\Corr(A,B)$.  

We prove (i).  Suppose that $f\colon X\to Y$ is a map between fibrant objects in $\Corr(A,B)$ such that 
$a^*X\to a^*Y$ is a covariant equivalence in $(\sSet)_{/(B^{\op}\times A)}$.    
We need to prove that $f$ is a categorical 
equivalence.  Therefore, we need to prove that $f$ is essentially 
surjective and fully faithful.  The essential surjectivity is immediate 
since $f$ is a map between correspondences in $\Corr(A,B)$.  To prove that 
$f$ is fully faithful it suffices to prove that the induced map on mapping spaces 
$\Hom^L_X(a,b)\to \Hom^L_Y(a,b)$ is a weak homotopy equivalence for each pair of objects 
$a\in A$ and $b\in B$.  This follows immediately from Remark~\ref{rem:mapping spaces} 
and Remark~\ref{rem:alt desc of a*}.    

Now we prove (ii).  Let $X\to B^{\op}\times A$ be an object of 
$(\sSet)_{/(B^{\op}\times A)}$; then we may factor $a_!X\to B\star A$ in $\Corr(A,B)$ as 
$a_!X\to \boldR a_!X\to B\star A$, where $a_!X\to \boldR a_!X$ is a categorical equivalence and 
$\boldR a_!X\to B\star A$ is an inner fibration.  Thus $\boldR a_!X$ is a fibrant replacement 
of $a_!X$ in the model structure for correspondences 
on $\Corr(A,B)$ (Theorem~\ref{thm:model str for corr}).  We claim that $a^*$ sends 
categorical equivalences in $\Corr(A,B)$  
to covariant equivalences in $(\sSet)_{/(B^{\op}\times A)}$.  To see this we argue as follows: 
suppose that $f\colon X\to Y$ is a categorical equivalence in $\Corr(A,B)$ and choose 
a fibrant replacement $Y\to Y'$ of $Y$ in $\Corr(A,B)$.  Then $Y'$ is an $\infty$-category, 
and $j\colon Y\to Y'$ is an acyclic cofibration which is a bijection on objects.  It follows 
that $j$ is inner anodyne (Lemma~\ref{lem:inner anodyne lemma}).  
We may factor the composite map $jf$ in $(\sSet)_{/B\star A} $ as $jf = f'j'$, 
where $f'\colon X'\to Y'$ is an inner fibration, and where $j'\colon X\to X'$ is inner anodyne.  
We observe that the underlying simplicial map $f'$ is a bijection on vertices, and hence is a 
categorical fibration.  Therefore, since $f'$ is a categorical equivalence, $f'$ is a 
trivial Kan fibration.  Hence $\sigma^*(f')$ is a trivial Kan fibration, since 
$\sigma_!$ preserves monomorphisms (Lemma~\ref{lem:sigma_! mono}).  It now suffices 
to prove the following claim: the functor $\sigma^*\colon (\sSet)_{/(B\star A)}\to (\sSet)_{/(B^{\op}\times A)}$ 
sends inner anodyne maps in $(\sSet)_{/B\star A}$ to left anodyne maps in $(\sSet)_{/(B^{\op}\times A)}$.

The inner anodyne maps in $(\sSet)_{/B\star A}$ are a saturated class of monomorphisms, 
generated by the inner horn inclusions $\Lambda^n_k\to \Delta^n$ in $(\sSet)_{/B\star A}$.  
We need to calculate the image $\sigma^*(\Lambda^n_k)\to \sigma^*(\Delta^n)$ of such a 
horn inclusion under the functor $\sigma^*$.  
The simplices in $B\star A$ are of the following three types: 
$x\star \emptyset\colon \Delta^n\star \emptyset\to B\star A$, $x\star y\colon \Delta^m\star \Delta^{n}\to B\star A$, 
and $\emptyset\star y\colon \emptyset\star \Delta^n\to B\star A$.  It follows that 
the inner horn inclusions in $(\sSet)_{/B\star A}$ are of the following types: 
\begin{itemize} 
\item $\Lambda^n_k\star \emptyset\to \Delta^n\star \emptyset$, $0<k<n$,  
\item $\Lambda^m_k\star \Delta^{n}\cup \Delta^m\star \partial\Delta^{n}\to \Delta^m\star \Delta^{n}$, $0<k\leq m$, $n\geq 0$, 
\item $\Delta^m\star \Lambda^n_k\cup \partial\Delta^m\star \Delta^n \to \Delta^m\star \Delta^{n}$, $m\geq 0$, $0\leq k<n$, 
\item $\emptyset\star \Lambda^n_k\to \emptyset\star \Delta^n$, $0<k<n$.  
\end{itemize} 
By Lemma~\ref{lem:sigma*} the image under $\sigma^*$ of each of 
the first and last of these types of morphism is the empty map, 
while the image under $\sigma^*$ of the second and third 
maps are respectively the left anodyne morphisms 
\begin{itemize}  
\item $(\Lambda^m_k)^{\op}\times \Delta^{n}\cup (\Delta^m)^{\op}\times 
\partial\Delta^{n}\to (\Delta^m)^{\op}\times \Delta^{n}$, $0<k\leq m$, $n\geq 0$, 
\item $(\Delta^m)^{\op}\times \Lambda^n_k\cup (\partial\Delta^{m})^{\op}\times 
\Delta^{n}\to (\Delta^{m})^{\op}\times \Delta^{n}$, $m\geq 0$, $0\leq k<n$.  
\end{itemize}  
This completes the proof of the claim.

To complete the proof of the theorem it suffices to prove that $X\to a^*a_!X$ is a covariant equivalence 
in $(\sSet)_{/(B^{\op}\times A)}$.  Equivalently, by Remark~\ref{rem:calculation of unit map}, 
it suffices to prove that $X\to \sigma^*\sigma_!X$ is a covariant equivalence 
in $(\sSet)_{/(B^{\op}\times A)}$.  We will prove that in fact this map is a left anodyne 
map in $(\sSet)_{/(B^{\op}\times A)}$.

Using the skeletal filtration of $X$, we see that by an induction argument we are 
reduced to the case where $X$ is obtained from $X'$ by adjoining a single $n$-simplex along an attaching 
map $\partial\Delta^n\to X'$.  We have a commutative diagram 
\[
\begin{tikzcd} 
\Delta^n \arrow[d] & \partial\Delta^n \arrow[l] \arrow[d] \arrow[r] & X' \arrow[d] \\ 
\sigma^*\sigma_!\Delta^n & \arrow[l] \sigma^*\sigma_!\partial\Delta^n \arrow[r] & \sigma^*\sigma_!X' 
\end{tikzcd} 
\]
in which the two right hand vertical maps are left anodyne by the inductive hypothesis, and where the 
left hand vertical map is the diagonal inclusion $\Delta^n\to \Delta^n\times 
\Delta^n$ (see Remark~\ref{rem:calculation of unit map}) 
and hence is left anodyne.  Therefore it suffices by Lemma~\ref{lem:pullback lemma} below to prove that 
for any $n\geq 0$ the square 
\[
\begin{tikzcd} 
\partial\Delta^n \arrow[d] \arrow[r] & \sigma^*\sigma_!\partial\Delta^n \arrow[d] \\ 
\Delta^n \arrow[r] & \sigma^*\sigma_!\Delta^n 
\end{tikzcd}
\]
is a pullback.  
From Remark~\ref{rem:calculation of unit map}, the map $\Delta^n\to \sigma^*\sigma_!\Delta^n$ is the diagonal 
inclusion $\Delta^n\to \Delta^n\times \Delta^n$.  Let us write $\delta_n\colon \Delta^n\to 
\Delta^n\times \Delta^n$ for this map.  Clearly the square 
\[
\begin{tikzcd} 
\Delta^{n-1} \arrow[d,"d_i"'] \arrow[r,"\delta_{n-1}"] & \Delta^{n-1}\times \Delta^{n-1} \arrow[d,"d_i\times d_i"] \\ 
\Delta^n \arrow[r,"\delta_n"'] & \Delta^n\times \Delta^n 
\end{tikzcd} 
\]
is a pullback for any $0\leq i\leq n$.  It follows that the square 
\begin{equation}
\label{eq:second square}
\begin{tikzcd} 
\partial_i\Delta^{n-1} \arrow[d] \arrow[r] & \partial_i\Delta^{n-1}\times \partial_i\Delta^{n-1} \arrow[d] \\ 
\Delta^n \arrow[r,"\delta_n"'] & \Delta^n\times \Delta^n 
\end{tikzcd} 
\end{equation}
is a pullback for any $0\leq i\leq n$.  Since 
\[
\partial\Delta^n = \bigcup^n_{i=0}\partial_i\Delta^{n-1} 
\]
is a union of the subobjects $\partial_i\Delta^{n-1}$ of $\Delta^n$ and the functor $\sigma^*\sigma_!$ is a 
left adjoint, it follows that 
\[
\sigma^*\sigma_!\partial\Delta^n = \bigcup^n_{i=0} \partial_i\Delta^{n-1}\times \partial_i\Delta^{n-1} 
\]
is a union of the subobjects $\partial_i\Delta^{n-1}\times \partial_i\Delta^{n-1}$ of 
$\Delta^n\times \Delta^n$.  The result then follows from the fact that the 
square~\eqref{eq:second square} above is a pullback for every $0\leq i\leq n$.    
\end{proof}

\begin{lemma} 
\label{lem:pullback lemma} 
Suppose that 
\[
\begin{tikzcd}
A \arrow[d,"f"] & B \arrow[l] \arrow[r] \arrow[d,"g"] & C \arrow[d,"h"] \\ 
A' & \arrow[l] B' \arrow[r] & C' 
\end{tikzcd} 
\]
is a commutative diagram of maps of simplicial sets in which the left hand square is a 
pullback and in which the maps $f$, $g$ and 
$h$ are left anodyne.  Then the induced map 
\[
A\cup_BC\to A'\cup_{B'}C' 
\]
is also left anodyne.  
\end{lemma} 

\begin{proof} 
The induced map factors as 
\[
A\cup_B C\to A\cup_B C' \approx A\cup_B B'\cup_{B'} C'\to A'\cup_{B'} C' 
\]
and, since left anodyne maps are preserved under pushouts, we see that 
it suffices to prove that $A\cup_B B'\to A'$ is left anodyne.  
The map $A\to A'$ factors as $A\to A\cup_B B'\to A'$; therefore the result 
follows from the right cancellation property of left anodyne maps in 
$\sSet$ (Corollary 4.1.2.2 of \cite{HTT}), since $A\to A'$ is left anodyne by hypothesis 
and $A\to A\cup_B B'$ is a pushout of the left anodyne map $B\to B'$.  
\end{proof} 

The functor $a^*$ has the  
distinction of being simultaneously a left 
and right Quillen equivalence.  Recall  the 
adjoint pair $(a^*,a_*)$ (see Remark~\ref{rem:a* right adjoint}).  We have the following theorem.  

\begin{theorem} 
\label{thm:a^* a_* quillen equiv}
Let $A$ and $B$ be $\infty$-categories.  Then the adjoint pair 
\[
a^*\colon \Corr(A,B)\rightleftarrows (\sSet)_{/(B^{\op}\times A)}\colon a_* 
\]
is a Quillen equivalence for the model structure for correspondences on 
$\Corr(A,B)$ and the covariant model structure on $(\sSet)_{/(B^{\op}\times A)}$.  
\end{theorem} 

\begin{proof} 
We show first that the pair $(a^*,a_*)$ is a Quillen adjunction.  Clearly 
$a^*$ sends monomorphisms to monomorphisms; and we have proved above (see the 
proof of Theorem~\ref{thm: (delta_!,d) is a Quillen equivalence}) that 
$a^*$ sends categorical equivalences to covariant equivalences.  

To prove that the Quillen pair $(a^*,a_*)$ is a Quillen equivalence it suffices to prove 
that the Quillen pair $(a^*a_!,a^*a_*)$ is a Quillen equivalence.  We have proven 
above (see the proof of Theorem~\ref{thm: (delta_!,d) is a Quillen equivalence}) that 
the natural transformation $X\to a^*a_!X$ is left anodyne for every $X\in (\sSet)_{/(A^{\op}
\times B)}$.  It follows easily that $a^*a_!$ reflects covariant equivalences.  To 
complete the proof it suffices to prove that $a^*a_*X\to X$ is a covariant equivalence 
for every left fibration $X\to B^{\op}\times A$.  We will prove that in fact this 
map is a trivial Kan fibration.  Suppose given a commutative diagram 
\[
\begin{tikzcd} 
\partial\Delta^n\arrow[r] \arrow[d] & a^*a_*X \arrow[d] \\ 
\Delta^n \arrow[r] \arrow[ur,dashed] & X 
\end{tikzcd} 
\]
where $n\geq 0$.  By adjointness, the indicated diagonal filler in this diagram 
exists if and only if the indicated diagonal filler in the 
corresponding diagram 
\[
\begin{tikzcd} 
a^*a_!\partial\Delta^n\cup_{\partial\Delta^n}\Delta^n \arrow[r] \arrow[d] & X \arrow[d] \\ 
a^*a_!\Delta^n \arrow[r] \arrow[ur,dashed] & B^{\op}\times A 
\end{tikzcd} 
\]
exists.  But the left hand vertical map is left anodyne by the right cancellation property 
of the class of left anodyne maps 
(Proposition~\ref{prop:right cancellation for left anodynes}).  
Therefore the indicated diagonal filler exists since 
$X\to B^{\op}\times A$ is a left fibration.  
\end{proof}

\section{bifibrations} 
\label{sec:two sided fibns}
\subsection{The category of bifibrations} 
\label{subsec:categ of 2sided fibns}
In this section we recall the definition of bifibration from Section 2.4.7 of \cite{HTT}.

\begin{definition} 
\label{def:bifibration}
Let $A$ and $B$ be simplicial sets.  A map $(p,q)\colon X\to A\times B$ in $\sSet$ is called a 
{\em bifibration} if the following conditions are satisfied: 
\begin{enumerate} 
\item $(p,q)$ is an inner fibration; 
\item for every $n\geq 1$ and for every commutative diagram 
\[
\begin{tikzcd} 
\Lambda^n_0 \arrow[d] \arrow[r] & X \arrow[d,"{\left( p,q\right)}"] \\ 
\Delta^n \arrow[r,"f"'] \arrow[ur,dashed] & A\times B 
\end{tikzcd} 
\]
such that $\pi_Bf\colon \Delta^{\set{0,1}}\to B$ is a degenerate edge, the indicated diagonal 
filler exists; 

\item for every $n\geq 1$ and for every commutative diagram 
\[
\begin{tikzcd} 
\Lambda^n_n \arrow[d] \arrow[r] & X \arrow[d,"{\left( p,q\right)}"] \\ 
\Delta^n \arrow[r,"f"'] \arrow[ur,dashed] & A\times B 
\end{tikzcd} 
\]
such that $\pi_Af\colon \Delta^{\set{n-1,n}}\to A$ is a degenerate edge, the indicated diagonal 
filler exists.
\end{enumerate}
\end{definition} 

\begin{remark} 
\label{rem:examples of bifibrations}
Suppose that $p\colon X\to A$ is a left fibration where 
$X$ and $A$ are simplicial sets.  Then for any simplicial set $B$, the induced map 
$p\times \id_B\colon X\times B\to A\times B$ is a bifibration. The conditions (1) and 
(2) from Definition~\ref{def:bifibration} are clearly satisfied.  To see that condition 
(3) is satisfied, note that it suffices to prove that the indicated diagonal 
filler exists in every diagram of the form 
\[
\begin{tikzcd} 
\Lambda^{n}_n \arrow[r,"u"] \arrow[d] & X \arrow[d,"p"] \\ 
\Delta^n \arrow[r,"v"] \arrow[ur,dashed] & A 
\end{tikzcd} 
\]
in which $v|\Delta^{\set{n-1,n}}\to A$ is a degenerate edge of $A$.  
The existence of such a diagonal filler is clear when $n=1$; if $n\geq 2$ the 
existence of such a diagonal filler is equivalent to the existence of the 
indicated diagonal filler in the induced diagram 
\[
\begin{tikzcd} 
\Delta^{\set{n}} \arrow[d] \arrow[r] & X_{\Delta^{n-2}/} \arrow[d] \\ 
\Delta^{\set{n-1,n}} \arrow[r] \arrow[ur,dashed] & 
X_{\partial\Delta^{n-2}/}\times_{A_{\partial\Delta^{n-2}/}}A_{\Delta^{n-2}/} 
\end{tikzcd} 
\]
But the right hand vertical map in this diagram is a trivial 
Kan fibration (Proposition 2.1.2.5 of \cite{HTT}) 
since $p\colon X\to A$ is a left fibration.  The existence 
of the required diagonal filler follows.  
Dually, if $q\colon Y\to B$ is a right fibration, then for any simplicial 
set $A$, the induced map $\id_A\times q \colon A\times Y\to A\times B$ is a 
bifibration.    
\end{remark} 

\begin{lemma} 
\label{lem:fibers of bifib kan} 
Let $(p,q)\colon X\to A\times B$ be a bifibration, where $A$ and $B$ are simplicial sets.  
Then the fiber $X_{(a,b)}$ of $(p,q)$ over $(a,b)$ 
is a Kan complex for every pair of vertices $(a,b)\in A\times B$.  
\end{lemma} 

\begin{proof} 
Clearly bifibrations over $A\times B$ are stable under base change along maps of the form 
$f\times g\colon A'\times B'\to A\times B$.  Therefore, 
pulling back along the map $a\times b\colon \Delta^{\set{0}}\times 
\Delta^{\set{0}}\to A\times B$ induces a bifibration $X_{(a,b)}\to 
\Delta^{\set{0}}\times \Delta^{\set{0}}$.  It follows (Remark 2.4.7.4 of \cite{HTT}) 
that $X_{(a,b)}\to \Delta^{\set{0}}$ is a right fibration.  Hence $X_{(a,b)}$ is a 
Kan complex.  
\end{proof} 

\subsection{Bivariant anodyne maps} 
\label{subsec:biv anodyne}
In this section we introduce the concept of 
bivariant anodyne maps and study some of their 
properties.  

\begin{definition} 
\label{def:biv anodyne} 
Let $A$ and $B$ be simplicial sets.  
Let us say that a map $u\colon M\to N$ in $(\sSet)_{/(A\times B)}$ 
is a {\em bivariant anodyne} map if it belongs to the weakly saturated class generated 
by the following classes of maps in $(\sSet)_{/(A\times B)}$: 
\begin{enumerate} 
\item the inner horn inclusions 
\[
\begin{tikzcd}[column sep=small] 
\Lambda^n_i \arrow[dr] \arrow[rr,hook] & & \Delta^n \arrow[dl] \\ 
& A\times B & 
\end{tikzcd} 
\]
where $0<i<n$; 
\item the horn inclusions 
\[
\begin{tikzcd}[column sep=small] 
\Lambda^n_0\arrow[dr] \arrow[rr,hook] & &  \Delta^n \arrow[dl,"{(f,g)}"] \\ 
& A\times B & 
\end{tikzcd} 
\]
where $f\colon \Delta^{\set{0,1}}\to B$ 
is a degenerate edge; 
\item the horn inclusions 
\[
\begin{tikzcd}[column sep=small]
\Lambda^n_n\arrow[dr] \arrow[rr,hook] & &  \Delta^n \arrow[dl,"{(f,g)}"] \\
& A\times B & 
\end{tikzcd} 
\]
where $g\colon \Delta^{\set{n-1,n}}\to A$ 
is a degenerate edge.  
\end{enumerate}
\end{definition} 

We extend the original usage of the term bifibration in \cite{HTT} to cover 
the following more general situation.  

\begin{definition} 
\label{def:general bifibn}
Let $A$ and $B$ be simplicial sets and let $f\colon X\to Y$ be a map in 
$(\sSet)_{/(A\times B)}$.  We say $f$ is a {\em bifibration} if it has the right 
lifting property against all bivariant anodyne maps in $(\sSet)_{/(A\times B)}$.  
\end{definition} 

\begin{remark} 
\label{rem:char of bifibrations} 
Let $X$ and $Y$ be objects of $(\sSet)_{/(A\times B)}$ and suppose that 
$X$ has structure map $(p,q)\colon X\to A\times B$.  A map 
$f\colon X\to Y$ in 
$(\sSet)_{/(A\times B)}$ 
is a bifibration if and only if the following conditions are 
satisfied: 

\begin{enumerate} 
\item $f$ is an inner fibration; 

\item for every commutative diagram 
\[
\begin{tikzcd} 
\Lambda^n_0 \arrow[r,"u"] \arrow[d] & X \arrow[d,"f"] \\ 
\Delta^n \arrow[r] \arrow[ur,dashed] & Y 
\end{tikzcd} 
\]
in which the edge $qu\colon \Delta^{\set{0,1}}\to B$ is degenerate, the indicated 
diagonal filler exists; 

\item for every commutative diagram 
\[
\begin{tikzcd} 
\Lambda^n_n \arrow[d] \arrow[r,"v"] & X \arrow[d,"f"] \\ 
\Delta^n \arrow[r] \arrow[ur,dashed] & Y 
\end{tikzcd} 
\]
in which the edge $pv\colon \Delta^{\set{n-1,n}}\to A$ is degenerate, the indicated 
diagonal filler exists.  
\end{enumerate} 

To see the equivalence of these statements note that if $f\colon X\to Y$ is 
a bifibration then the conditions (i), (ii) and (iii) are automatically satisfied.  
Conversely, suppose that $f\colon X\to Y$ is a map in $(\sSet)_{/(A\times B)}$ such that 
(i), (ii) and (iii) are satisfied.  The class of maps in $(\sSet)_{/(A\times B)}$ 
which have the right lifting property against $f$ is weakly saturated.  By hypothesis 
it contains the classes of maps (1), (2) and (3) from Definition~\ref{def:biv anodyne}; 
hence it contains the weakly saturated class generated by these maps, in other words 
the class of bivariant anodyne maps.  
\end{remark} 
 
\begin{remark}
Thus a map $X\to A\times B$ is a bifibration in the sense of Definition~\ref{def:general bifibn} 
if and only if it is a bifibration in the sense of Definition~\ref{def:bifibration}.  
\end{remark} 

\begin{remark} 
\label{rem:restriction of 2-sided fibn}
If $X\to Y$ is a bifibration in $(\sSet)_{/(A\times B)}$ then 
for every vertex $b$ in $B$ the restriction $X|A\times \set{b}\to Y|A\times 
\set{b}$ is a left fibration in $(\sSet)_{/A}$.  Similarly for every 
vertex $a$ in $A$ the restriction $X|\set{a}\times B\to Y|\set{a}\times B$ 
is a right fibration in $(\sSet)_{/B}$.  
\end{remark}

\begin{proposition} 
\label{prop:alt desc biv anodyne}
Let $A$ and $B$ be simplicial sets.  
Consider the following classes of morphisms in $(\sSet)_{/(A\times B)}$: 

(1) all inner horn inclusions 
\[
\begin{tikzcd}[column sep =small] 
\Lambda^n_i \arrow[dr] \arrow[rr,hook] & &  \Delta^n \arrow[dl] \\ 
& A\times B & 
\end{tikzcd} 
\]
for $0<i<n$; 

(2) all horn inclusions 
\[
\begin{tikzcd} 
\Lambda^n_0 \arrow[dr] \arrow[rr,hook] & &  \Delta^n \arrow[dl,"{(f,g)}"] \\ 
& A\times B & 
\end{tikzcd}  
\]
such that $g|\Delta^{\set{0,1}}$ 
is a degenerate edge of $B$; 

($2'$) all inclusions 
\[
\begin{tikzcd}[column sep=small] 
\Delta^1\times \partial\Delta^n\cup \set{0}\times \Delta^n \arrow[dr] \arrow[rr,hook] & &  \Delta^1\times \Delta^n \arrow[dl,"{(f,g)}"] \\ 
& A\times B & 
\end{tikzcd} 
\]
such that $g|\Delta^1\times \set{i}$ is a degenerate edge of $B$ for every vertex $i$ in $\Delta^n$; 

($2''$) all inclusions of the form 
\[
\begin{tikzcd}[column sep=small] 
\Delta^1\times K \cup \set{0}\times L \arrow[dr] \arrow[rr,hook] & & \Delta^1\times L \arrow[dl,"{(f,g)}"] \\ 
& A\times B & 
\end{tikzcd} 
\]
where $K\hookrightarrow L$ is an inclusion and where $g|\Delta^1\times \set{v}$ is a degenerate edge 
of $B$ for every vertex $v$ of $L$.  

Then the weakly saturated classes of morphisms in $(\sSet)_{/(A\times B)}$ generated by the classes (1) and (2), the classes 
(1) and ($2'$), and the classes (1) and ($2''$) are all equal. 
\end{proposition}

\begin{proof} 
The proof of this proposition is essentially the same as the proof of Proposition 3.1.1.5 
of \cite{HTT}.  We give the details.  To begin with, the weakly saturated class generated 
by (1) and ($2'$) is clearly contained in the weakly saturated class generated by (1) and 
($2''$).  As in the proof of op.\ cit., one easily 
proves that the weakly saturated class generated by 
(1) and ($2''$) is contained in the weakly saturated class 
generated by (1) and ($2'$).  It follows that the weakly 
saturated class generated by (1) and ($2'$) is equal to the 
weakly saturated class generated by (1) and ($2''$).  We prove that 
every map in (2) is a retract of a map in ($2''$).  Suppose given a map 
\[
\begin{tikzcd}[column sep=small] 
\Lambda^n_0 \arrow[rr,hook] \arrow[dr] & & \Delta^n \arrow[dl,"{(f,g)}"] \\ 
& A\times B 
\end{tikzcd} 
\]
such that $g|\Delta^{\set{0,1}}$ is a degenerate edge of $B$.  Let $j\colon \Delta^n\to \Delta^n\times 
\Delta^1$ correspond to the inclusion $\Delta^n\times \set{1}\subseteq \Delta^n\times \Delta^1$.  
Define the retraction $r\colon \Delta^n\times \Delta^1\to \Delta^n$ 
in $(\sSet)_{/(A\times B)}$ as the map induced by 
the map $r\colon [n]\times [1]\to [n]$ of partially ordered sets defined by 
\[
r(m,0) = \begin{cases} m & \text{if}\ m\neq 1, \\ 0 & \text{if}\ m=1 \end{cases} 
\]
and by $r(m,1) = m$ for all $m\in [n]$.  Observe that the composite map 
$gr\colon \Delta^n\times \Delta^1\to B$ restricts to a degenerate edge  
$gr|\set{i}\times \Delta^1$ of $B$ for every vertex $i$ of $\Delta^n$; this is clear 
from the definition of $r$ if $i\neq 1$ and follows from the assumption that $g|\Delta^{\set{0,1}}$ 
is degenerate when $i=1$.  The maps $j$ and $r$ exhibit the inclusion $\Lambda^n_0
\hookrightarrow \Delta^n$ as a retract of the map 
\[
\Lambda^n_0\times \Delta^1 \cup \Delta^n\times \set{0}\hookrightarrow \Delta^n\times 
\Delta^1 
\]
in $(\sSet)_{/(A\times B)}$ with structure map $(f,g)r\colon \Delta^n\times \Delta^1\to A\times B$.  
From the discussion above, this map belongs to the class of maps ($2''$).  It follows that 
the weakly saturated class generated by (1) and (2) is contained in the weakly saturated 
class generated by (1) and ($2''$).    

We now prove that the weakly saturated class generated by (1) and ($2'$) is contained in the 
weakly saturated class generated by (1) and (2).  Suppose given a map in ($2'$) of the 
form 
\[
\begin{tikzcd}[column sep=small]
\partial\Delta^n\times \Delta^1 \cup \Delta^n\times \set{0} \arrow[dr] \arrow[rr,hook] 
& & \Delta^n \times \Delta^1 \arrow[dl,"{(f,g)}"] \\ 
& A\times B & 
\end{tikzcd} 
\]
in which the structure map $(f,g)\colon \Delta^n\to A\times B$ satisfies 
$g|\set{i}\times \Delta^1$ is a degenerate edge of $B$ for every vertex $i$ of 
$\Delta^n$.  We have the standard filtration 
\[
\partial\Delta^n\times \Delta^1 \cup \Delta^n\times \set{0} = 
X^0\subseteq X^1\subseteq \cdots \subseteq X^n = \Delta^n\times \Delta^1 
\]
in which the inclusion $X^{i}\subseteq X^{i+1}$ fits into a pushout diagram 
of the form 
\[
\begin{tikzcd} 
\Lambda^{n+1}_{i} \arrow[r] \arrow[d] & X^i \arrow[d] \\ 
\Delta^{n+1} \arrow[r] & X^{i+1} 
\end{tikzcd} 
\]
for every $i=0,1,\ldots,n$.  The $(n+1)$-simplex of $X^{n}$ obtained from $X^{n-1}$ 
via the attaching map $\Lambda^{n+1}_0\to X^{n-1}$ corresponds to the 
$(n+1)$-chain 
\[
\begin{tikzcd} 
(0,0) \arrow[d] & & & \\ 
(0,1) \arrow[r] & (1,1) \arrow[r] & \cdots \arrow[r] & (n,1) 
\end{tikzcd} 
\]
of $[n]\times [1]$.  By assumption the edge $g|\set{0}\times \Delta^1$ of 
$B$ is degenerate.  It follows easily that the map above belongs to the weakly 
saturated class generated by the maps (1) and (2).  
\end{proof}

\subsection{Stability properties for bifibrations} 
\label{subsec:stability for 2-sided}
In this section we prove some stability properties 
for bifibrations under exponentiation, analogous to the discussion in 
Section 2.1.2 of \cite{HTT} for left fibrations.  

\begin{proposition} 
\label{prop:exp}
Let $f\colon X\to Y$ be a bifibration in $(\sSet)_{/(A\times B)}$.  
Then for any monomorphism $u\colon M\to N$ in 
$(\sSet)_{/(A\times B)}$, the induced map 
\[
X^N\to X^M\times_{Y^M}Y^N 
\]
is a bifibration in $(\sSet)_{/(A^N\times B^N)}$.  
\end{proposition} 

\begin{proof} 
The induced map $X^N\to X^M\times_{Y^M}Y^N$ is an inner fibration by Corollary 2.3.2.5.\  of \cite{HTT}.
We prove that the induced map has the right lifting property 
against the class of maps (2) from Definition~\ref{def:biv anodyne}.  
By Proposition~\ref{prop:alt desc biv anodyne} it suffices to 
prove that the indicated diagonal filler exists in every commutative diagram of the form 
\[
\begin{tikzcd} 
\Delta^1\times \partial\Delta^n\cup \set{0}\times \Delta^n \arrow[r] \arrow[d] & X^N \arrow[d] \\ 
\Delta^1\times \Delta^n \arrow[r] \arrow[ur,dashed] & X^M\times_{Y^M}Y^N 
\end{tikzcd} 
\]
where, if $u\colon \Delta^1\times \Delta^n\to X^M\times_{Y^M}Y^N\to A^N\times B^N$ denotes 
the composite map, then $\pi_{B^N}u|\Delta^1\times \set{i}$ is a degenerate edge in $B^N$ for 
every vertex $i$ of $\Delta^n$.  By adjointness, 
it is sufficient to prove that the indicated diagonal filler exists in the commutative 
diagram 
\[
\begin{tikzcd} 
\Delta^1\times (\partial\Delta^n\times N \cup 
\Delta^n\times M)\cup \set{0}\times \Delta^n\times N 
\arrow[d] \arrow[r] & X \arrow[d,"f"] \\ 
\Delta^1\times \Delta^n\times N \arrow[r] \arrow[ur,dashed] & Y. 
\end{tikzcd} 
\]
For every vertex $i$ of $\Delta^n$ and for every vertex $v$ of $N$, the map $\Delta^1\times 
\Delta^n\times N\to Y$ restricts to an edge $\Delta^1\times \set{i}\times \set{n}\to Y$ 
of $Y$ which is mapped to a degenerate edge in $B$ by $q\colon Y\to B$, where 
$(p,q)\colon Y\to A\times B$ denotes the structure map. 
Therefore the indicated diagonal filler exists by 
Proposition~\ref{prop:alt desc biv anodyne}.  

The proof that the induced map has the right lifting property against all maps in 
the class (3) of Definition~\ref{def:biv anodyne} is completely analogous.  
\end{proof}

Let $X$ and $M$ be objects of $(\sSet)_{/(A\times B)}$.  
Recall (Notation~\ref{not:simp mapping spaces}) that 
$\map_{A\times B}(M,X)$ denotes the simplicial mapping 
space for the simplicially enriched category  
$(\sSet)_{/(A\times B)}$.   

\begin{lemma} 
\label{cor:map AxB Kan cplx}
Let $p\colon X\to Y$ be a 
bifibration in 
$(\sSet)_{/(A\times B)}$.  
Then for any 
monomorphism $u\colon K\to L$ 
in $(\sSet)_{/(A\times B)}$, the 
induced map 
\[
\map_{A\times B}(u,p)\colon \map_{A\times B}(L,X) \to 
\map_{A\times B}(K,X)\times_{ 
\map_{A\times B}(K,Y)} 
\map_{A\times B}(L,Y) 
\] 
is a Kan fibration 
between Kan complexes.  
\end{lemma} 

\begin{proof} 
We first prove that 
$\map_{A\times B}(M,X)$ is a 
Kan complex for any bifibration 
$X\to A\times B$ and any object 
$M$ in $(\sSet)_{/(A\times B)}$.  
By Proposition~\ref{prop:exp} the induced map 
$X^M\to A^M\times B^M$ is a bifibration.  
We have a pullback diagram 
\[
\begin{tikzcd} 
\map_{A\times B}(M,X) \arrow[d] \arrow[r] & X^M \arrow[d] \\ 
\Delta^0 \arrow[r,"\phi"] & A^M\times B^M 
\end{tikzcd} 
\]
where $\phi$ corresponds to the structure 
map $M\to A\times B$.  Hence $\map_{A\times B}(M,X)$ 
is a Kan complex by Lemma~\ref{lem:fibers of bifib kan}.    

Next we prove the assertion in the special 
case that $p$ is a  
bifibration $X\to A\times B$.  
Suppose that $u\colon K\to L$ is a monomorphism 
in $(\sSet)_{/(A\times B)}$.  The induced map 
\[
\map_{A\times B}(L,X)\to 
\map_{A\times B}(K,X) 
\]
is an inner fibration.  Since it is an inner fibration 
between Kan complexes it suffices by 
Lemma~\ref{lem:inner implies kan} 
to prove that it has the right lifting 
property with respect to the inclusion 
$\Delta^{\set{0}}\subseteq \Delta^1$.
By adjointness, the indicated diagonal filler 
exists in the diagram 
\[
\begin{tikzcd} 
\Delta^{\set{0}} \arrow[r] 
\arrow[d] & \map_{A\times B}(L,X) \arrow[d] \\ 
\Delta^1 \arrow[r] \arrow[ur,dashed] & \map_{A\times B}(K,X) 
\end{tikzcd} 
\]
if and only if the indicated diagonal 
filler exists in the diagram 
\[
\begin{tikzcd} 
L\times \Delta^{\set{0}}\cup 
K\times \Delta^1 
\arrow[r] \arrow[d] & X \arrow[d] \\ 
L\times \Delta^1 \arrow[r] 
\arrow[ur,dashed] & A\times B 
\end{tikzcd} 
\]
But for any vertex $v$ of $L$, the map 
$\set{v}\times \Delta^1\to A\times B$ 
is sent to a degenerate edge of $B$ under 
the projection $A\times B\to B$.  The indicated 
diagonal fillers therefore exist by 
Proposition~\ref{prop:alt desc biv anodyne}. 
It follows by Lemma~\ref{lem:inner implies kan} 
that the induced map above is a Kan fibration 
between Kan complexes.  

Finally, we prove the general form of the 
assertion.  Suppose that $p\colon X\to Y$ is a 
bifibration and that $u\colon K\to L$ 
is a monomorphism.  We use 
Lemma~\ref{lem:inner implies kan} again.  The map 
\[
\map_{A\times B}(L,X) \to 
\map_{A\times B}(K,X)\times_{ 
\map_{A\times B}(K,Y)} 
\map_{A\times B}(L,Y) 
\] 
is an inner fibration between Kan complexes by 
the results of the preceding paragraphs.  
Therefore we are reduced to proving that the 
indicated diagonal filler exists in any 
commutative diagram of the form 
\[
\begin{tikzcd} 
\Delta^{\set{0}} \arrow[r] \arrow[d] & 
\map_{A\times B}(L,X) \arrow[d] \\ 
\Delta^1 \arrow[r] \arrow[ur,dashed] & 
\map_{A\times B}(K,X)\times_{ 
\map_{A\times B}(K,Y)} 
\map_{A\times B}(L,Y) 
\end{tikzcd} 
\]
By adjointess, this is equivalent to 
proving that the indicated diagonal 
filler exists in the induced diagram 
\[
\begin{tikzcd} 
L\times \Delta^{\set{0}} 
\cup K\times \Delta^1 
\arrow[d] \arrow[r] & 
X \arrow[d,"p"]            \\ 
L\times \Delta^1 
\arrow[ur,dashed] \arrow[r] 
& Y 
\end{tikzcd} 
\]
By Proposition~\ref{prop:alt desc biv anodyne}, 
using the fact that the composite map 
$L\times \Delta^1\to Y\to A\times B$ factors 
through $L\times \Delta^{0}$ via the 
structure map $L\to A\times B$, we see 
that such a diagonal filler exists.  
This completes the proof of the Lemma.  
\end{proof} 

\subsection{The bivariant model structure} 
\label{subsec:biv model str} 
In this section we 
describe the model structure for bifibrations 
(see Theorem~\ref{thm:biv model str}).

\begin{definition} 
A map $Y\to Z$ in $(\sSet)_{/(A\times B)}$ is said to be a {\em bivariant equivalence} 
if the induced map 
\[
\map_{A\times B}(Z,X)\to \map_{A\times B}(Y,X) 
\]
is a homotopy equivalence between Kan complexes for every bifibration $X\to A\times B$. 
\end{definition} 

\begin{lemma} 
\label{lem:biv ano implies biv eq}
Suppose that $u\colon K\to L$ is a bivariant anodyne map.  Then $u$ is a bivariant equivalence.  
\end{lemma} 

\begin{proof} 
Let $\mathcal{A}$ denote 
the class of monomorphisms 
$u\colon K\to L$ in $(\sSet)_{/(A\times B)}$ such that the induced map 
\[
\map_{A\times B}(L,X) \to \map_{A\times B}(K,X) 
\]
is a homotopy equivalence for all bifibrations 
$X\to A\times B$. By 
Lemma~\ref{cor:map AxB Kan cplx}, the class 
$\mathcal{A}$ is equivalently the class of 
monomorphisms $u\colon K\to L$ in 
$(\sSet)_{/(A\times B)}$ such that the induced 
map above is a trivial Kan fibration   
for every bifibration $X\to A\times B$. 
It follows easily that $\mathcal{A}$ is weakly saturated. 
 
To complete the proof we will prove that $\mathcal{A}$ contains the classes (1), (2) and (3) from 
Definition~\ref{def:biv anodyne}.  It is clear that $\mathcal{A}$ contains the class of inner 
anodyne maps in $(\sSet)_{/(A\times B)}$.  We prove that $\mathcal{A}$ contains the class of maps 
(2) from Definition~\ref{def:biv anodyne} (the proof that $\mathcal{A}$ contains the class of 
maps (3) from Definition~\ref{def:biv anodyne} is completely analogous).  

It suffices to prove 
that $\mathcal{A}$ contains the class of maps ($2''$) from Proposition~\ref{prop:alt desc biv anodyne}.  
Let $X\to A\times B$ be a bifibration.  Let $u\colon M\to N$ belong to the class of 
maps ($2''$) from Proposition~\ref{prop:alt desc biv anodyne}.  We prove that 
the induced map 
\[
\map_{A\times B}(N,X)\to \map_{A\times B}(M,X) 
\]
is a trivial Kan fibration.   
Consider a commutative diagram 
\[
\begin{tikzcd} 
\partial\Delta^n \arrow[r] \arrow[d] & \map_{A\times B}(N,X) \arrow[d] \\ 
\Delta^n \arrow[r] \arrow[ur,dashed]& \map_{A\times B}(M,X) 
\end{tikzcd}
\]
To show that the indicated diagonal filler in this diagram exists, it suffices, by adjointness, 
to prove that $X\to A\times B$ has the right lifting property against 
the canonical map 
\begin{equation} 
\label{eq:pushout product}
\partial\Delta^n\times N\cup \Delta^n\times M\to \Delta^n\times N  
\end{equation}
in $(\sSet)_{/(A\times B)}$ where the structure map $\Delta^n\times N\to A\times B$ factors 
as $\Delta^n\times N\xrightarrow{p_2} N\to A\times B$, and where $N\to A\times B$ is 
the given structure map of the object $N$ of $(\sSet)_{/(A\times B)}$.  
It follows easily that the map~\eqref{eq:pushout product} belongs to the class of maps 
($2''$) from Proposition~\ref{prop:alt desc biv anodyne} and hence the 
indicated diagonal filler can be found. 
\end{proof} 

Recall that a {\em fiberwise homotopy} 
between maps $f,g\colon X\to Y$ in $(\sSet)_{/(A\times B)}$ 
is an edge in the simplicial set 
$\map_{A\times B}(X,Y)$ (see Notation~\ref{not:simp mapping spaces})
between the vertices $f$ and $g$.  Recall that a map $h\colon X\to Y$ in 
$(\sSet)_{/(A\times B)}$ is said to be a {\em fiberwise homotopy equivalence} 
if there exists a map $k\colon Y\to X$ in $(\sSet)_{/(A\times B)}$ such that 
the maps $hk, 1_Y$ and the maps $kh,1_X$ are fiberwise homotopic.      

We have the following result.  

\begin{lemma}
\label{lem:fiberwise hty equiv} 
If $f\colon X\to Y$ is a fiberwise homotopy equivalence in $(\sSet)_{/(A\times B)}$ 
then $f$ is a bivariant equivalence.  If $X\to A\times B$ and $Y\to A\times B$ are 
bifibrations then the converse is also true.  
\end{lemma} 

\begin{proof} 
To prove the first statement it suffices to prove that if $h\colon X\times \Delta^1\to Y$ is a 
fiberwise homotopy between maps $f,g\colon X\to Y$ in $(\sSet)_{/(A\times B)}$, then $h$ induces 
a homotopy between the maps $f^*,g^*\colon \map_{A\times B}(Y,Z)\to \map_{A\times B}(X,Z)$ 
for any bifibration $Z\to A\times B$.  This follows easily from the fact that 
$\map_{A\times B}(M\times \Delta^1,Z) = \map_{A\times B}(M,Z)^{\Delta^1}$ for any 
object $M$ in $(\sSet)_{/(A\times B)}$.  

We prove the second statement.  Suppose that $f\colon X\to Y$ is a 
bivariant equivalence between bifibrations.  Observe that 
the map $f^*\colon \map_{A\times B}(Y,X)\to \map_{A\times B}(X,X)$ 
is a homotopy equivalence between Kan complexes and hence there exists 
a map $g\colon Y\to X$ in $(\sSet)_{/(A\times B)}$ and 
an edge $h$ in $\map_{A\times B}(X,X)$ between $gf$ and $1_X$.  
Hence $gf$ and $1_X$ are fiberwise homotopic.  The vertices $fgf$ and $f$ 
belong to the same path component in $\map_{A\times B}(X,Y)$.  Therefore, 
by the assumption on $f$, there exists an edge $k$ in $\map_{A\times B}(Y,Y)$ 
between the vertices $fg$ and $1_Y$.  Hence $fg$ and $1_Y$ are fiberwise homotopic.  Hence 
$f$ is a fiberwise homotopy equivalence.  
\end{proof} 

\begin{lemma} 
\label{lem:triv fibn biv equiv}
A trivial fibration in $(\sSet)_{/(A\times B)}$ is a bivariant equivalence.  
\end{lemma} 

\begin{proof}
This follows immediately from Lemma~\ref{lem:fiberwise hty equiv}, using 
the fact that a trivial fibration in $(\sSet)_{/(A\times B)}$ is a fiberwise 
homotopy equivalence.  
\end{proof} 

\begin{definition} 
A map $X\to Y$ in $(\sSet)_{/(A\times B)}$ is said to be a bivariant fibration if it has 
the right lifting property against all monic bivariant equivalences.  
\end{definition} 

\begin{proposition} 
\label{prop:char of bifibns between bifibns}
Let $f\colon X\to Y$ be a map in $(\sSet)_{/(A\times B)}$ between bifibrations 
$X\to A\times B$ and $Y\to A\times B$.  Then $f$ is a bifibration 
if and only if $f$ is a bivariant fibration.  
\end{proposition} 

\begin{proof} 
If $f\colon X\to Y$ is a bivariant fibration then it has the right lifting property 
against every bivariant anodyne map in $(\sSet)_{/(A\times B)}$ since a bivariant anodyne map 
is a bivariant equivalence (Lemma~\ref{lem:biv ano implies biv eq}).  

We prove the converse.  Suppose that $f\colon X\to Y$ has the right lifting property 
against every bivariant anodyne map and that $X\to A\times B$, $Y\to A\times B$ are 
bifibrations.  Let $M\to N$ be a monic bivariant equivalence.  We need to show 
that we can find the indicated diagonal filler in any commutative diagram of the form 
\[
\begin{tikzcd} 
M \arrow[r] \arrow[d] & X \arrow[d,"f"] \\ 
N \ar[r] \ar[ur,dashed] & Y 
\end{tikzcd} 
\]
From such a diagram we obtain the commutative diagram 
\begin{equation} 
\label{eq:3.17 one}
\begin{tikzcd} 
\map_{A\times B}(N,X) \arrow[d] \arrow[r] & \map_{A\times B}(N,Y) \arrow[d] \\ 
\map_{A\times B}(M,X) \arrow[r] & \map_{A\times B}(M,Y) 
\end{tikzcd} 
\end{equation}
in which each of the horizontal 
maps and vertical maps are Kan fibrations between 
Kan complexes by Lemma~\ref{cor:map AxB Kan cplx}.  
Since $M\to N$ is a bivariant equivalence, 
the vertical maps are in fact trivial Kan fibrations.  
It follows from Lemma~\ref{cor:map AxB Kan cplx} that the induced map 
\[
\map_{A\times B}(N,X) \to 
\map_{A\times B}(M,X) 
\times_{ \map_{A\times B}(M,Y)} 
\map_{A\times B}(N,Y) 
\]
is a trivial Kan fibration.  In particular it is surjective 
on vertices which implies the existence 
of the sought-after diagonal filler in the 
diagram above.  
\end{proof}

\begin{proposition} 
\label{prop:biv eqs access}
Let $A$ and $B$ be simplicial sets.  The subcategory of bivariant equivalences 
in the category of morphisms $((\sSet)_{/(A\times B)})^{[1]}$ is an accessible subcategory.  
\end{proposition} 

\begin{proof} 
We first prove that if a map $f\colon X\to Y$ in $(\sSet)_{/(A\times B)}$ is a bivariant 
fibration and a bivariant equivalence then it is a trivial fibration.  Given such a map 
$f$, we factor it as $f = pi$ where $i\colon X\to X'$ is a monomorphism and 
where $p\colon X'\to Y$ is a trivial fibration in $(\sSet)_{/(A\times B)}$.  Then we have 
a commutative diagram 
\[
\begin{tikzcd} 
X \arrow[d,"i"'] \arrow[r,"1_X"] & X \arrow[d,"f"] \\ 
X' \arrow[r,"p"'] \arrow[ur,dashed] & Y 
\end{tikzcd} 
\]
By Lemma~\ref{lem:triv fibn biv equiv}, the map $p$ is a bivariant equivalence, 
hence $i$ is a bivariant equivalence by 2-out-of-3.  Therefore the indicated 
diagonal filler exists, and hence $f$ is a retract of a trivial fibration.  
It follows that $f$ is a trivial fibration.  

The remainder of the proof proceeds in exactly the same fashion 
as the proof of Corollary A.2.6.6 of \cite{HTT}; the small object argument 
shows the existence of a functor $T\colon ((\sSet)_{/(A\times B)})^{[1]}
\to ((\sSet)_{/(A\times B)})^{[1]}$ together with a natural transformation 
$1\to T$ such that for any morphism $f\colon X\to Y$ in $(\sSet)_{/(A\times B)}$, in  
the diagram 
\[
\begin{tikzcd} 
X \arrow[d] \arrow[r,"f"] & Y \arrow[d] \\ 
T(X) \arrow[r,"T(f)"] & T(Y) 
\end{tikzcd} 
\]
the vertical maps are bivariant anodyne maps in $(\sSet)_{/(A\times B)}$, 
$T(f)\colon T(X)\to T(Y)$, $T(X)\to A\times B$ and $T(Y)\to A\times B$ 
are bifibrations.  Therefore $T(f)$ is a bivariant fibration.  
It follows that $f$ is a bivariant equivalence if and only if $T(f)$ is a trivial Kan fibration.  
Hence the bivariant equivalences in $((\sSet)_{/(A\times B)})^{[1]}$ form an accessible subcategory since 
the trivial Kan fibrations form an accessible subcategory.   
\end{proof} 

\begin{theorem} 
\label{thm:biv model str}
Let $A$ and $B$ be simplicial sets.  Then there is the structure of a left proper, combinatorial 
model category on $(\sSet)_{/(A\times B)}$ for which 
\begin{itemize} 
\item the class of cofibrations is the class of monomorphisms in $(\sSet)_{/(A\times B)}$; 
\item the class of fibrations is the class of bivariant fibrations in $(\sSet)_{/(A\times B)}$.  
\end{itemize} 
The fibrant objects are precisely the bifibrations $X\to A\times B$.  
\end{theorem} 

Following Joyal we call this model structure on $(\sSet)_{/(A\times B)}$ the {\em bivariant} model 
structure on $(\sSet)_{/(A\times B)}$.  

\begin{proof} 
We use Proposition A.2.6.8 from \cite{HTT}.  The category $(\sSet)_{/(A\times B)}$ is presentable, so therefore 
we need to verify the conditions (1)--(5) from the statement of that proposition.  The conditions (1) and (4) 
are clear; the condition (3) is Proposition~\ref{prop:biv eqs access} and condition (5) follows from 
Lemma~\ref{lem:triv fibn biv equiv}. Therefore it remains to prove that condition (2) holds.  It suffices to prove 
that if $i\colon M\to N$ is a morphism in $(\sSet)_{/(A\times B)}$ which has the left lifting property against every 
bivariant fibration in $(\sSet)_{/(A\times B)}$, then $i$ is a monic bivariant equivalence.  Via the small object argument 
we can find a commutative diagram 
\[
\begin{tikzcd} 
M \arrow[d,"i"'] \arrow[r,"f"] & X \arrow[d,"p"] \\ 
N \arrow[r,"g"'] \arrow[ur,dashed,"d"] & Y 
\end{tikzcd} 
\]
in which $f$ and $g$ are bivariant anodyne maps, and $p\colon X\to Y$, $X\to A\times B$, and $Y\to A\times B$ 
are bifibrations.  It follows that $p$ is a bivariant fibration (Proposition~\ref{prop:char of bifibns between bifibns}) 
and hence the indicated diagonal filler $d\colon N\to X$ exists.  Since $di = f$ is monic it follows that $i$ is monic.    
Since $f$ and $g$ are bivariant equivalences, $d$ is also a bivariant equivalence (this follows 
easily from the fact that if $u$ is a map of Kan complexes such that there exist maps $v$ and $w$ such that 
$vu$ and $uw$ are homotopy equivalences, then $u$ is also a homotopy equivalence).  
Thus the class of monic bivariant equivalences forms a weakly saturated class, completing the 
verification of condition (2).  The result follows.  
\end{proof} 

\subsection{Comparison with the covariant and contravariant model structures}
\label{subsec:comparison}
In this section we 
study some relationships between the bivariant 
model structure on $(\sSet)_{/(A\times B)}$ and 
the contravariant model structure on $(\sSet)_{/A}$ 
and the covariant model structure on $(\sSet)_{/B}$.  

Let $\pi_A\colon A\times B\to A$ and 
$\pi_B\colon A\times B\to B$ be the 
canonical projections.  Observe that 
the functor $\pi_A^*\colon (\sSet)_{/A}\to 
(\sSet)_{/(A\times B)}$ admits a left 
adjoint $(\pi_A)_!$ and a right adjoint 
$(\pi_A)_*$.  Similarly the functor 
$\pi_B^*\colon (\sSet)_{/B}\to (\sSet)_{/(A\times B)}$ 
admits a left adjoint $(\pi_B)_!$ and a 
right adjoint $(\pi_B)_*$.  
We have the following proposition.  

\begin{proposition} 
\label{prop:biv eqs and co and cont eqs}
Let $A$ and $B$ be simplicial sets.  Then the following statements are true: 
\begin{enumerate}
\item The adjunction 
\[
(\pi_{B})_!\colon (\sSet)_{/(A\times B)} \rightleftarrows (\sSet)_{/B}\colon \pi_{B}^* 
\]
is a Quillen adjunction for the bivariant model structure on $(\sSet)_{/(A\times B)}$ 
and the contravariant model structure on $(\sSet)_{/B}$; 
\item The adjunction 
\[
(\pi_{A})_!\colon (\sSet)_{/(A\times B)} \rightleftarrows (\sSet)_{/A}\colon \pi_{A}^* 
\]
is a Quillen adjunction for the bivariant model structure on $(\sSet)_{/(A\times B)}$ 
and the covariant model structure on $(\sSet)_{/A}$.  
\end{enumerate}   
\end{proposition} 

\begin{proof} 
We prove statement (1), the proof of statement (2) follows by duality.  
It is clear that $\pi_{B}^*$ sends trivial Kan fibrations in $(\sSet)_{/B}$ to 
trivial Kan fibrations in $(\sSet)_{/(A\times B)}$.  Therefore it suffices 
by Proposition~\ref{prop:char of bifibns between bifibns} 
and Remark~\ref{rem:examples of bifibrations}
to prove that if $X\to Y$ is a right fibration in $(\sSet)_{/B}$ between 
right fibrations $X\to B$ and $Y\to B$, then $A\times X\to A\times Y$ is a 
bifibration in $(\sSet)_{/(A\times B)}$.  Clearly $A\times X\to A\times Y$ satisfies (1) and 
(3) of Remark~\ref{rem:char of bifibrations}, therefore it suffices 
to prove that the indicated diagonal filler exists in 
every diagram of the form 
\[
\begin{tikzcd} 
\Lambda^{n}_{0} \arrow[d] \arrow[r] & A\times X \arrow[d] \\ 
\Delta^{n} \arrow[r] \arrow[dr,"{(f,g)}"'] \arrow[ur,dashed] & A\times Y \arrow[d] \\ 
& A\times B 
\end{tikzcd} 
\]
in which $g\colon \Delta^{\set{0,1}}\to B$ 
is a degenerate edge of $B$. 
Clearly it suffices to show that the indicated diagonal filler exists 
in the induced diagram 
\begin{equation} 
\label{eq:3.19 one} 
\begin{tikzcd} 
\Lambda^{n}_{0} \arrow[d] \arrow[r,"u"] & X \arrow[d] \\ 
\Delta^{n} \arrow[r,"v"] \arrow[ur,dashed] & Y 
\end{tikzcd} 
\end{equation} 
in $(\sSet)_{/B}$.    
If $n=1$ then the indicated diagonal filler in~\eqref{eq:3.19 one} 
exists since the arrow $v\colon \Delta^1\to Y$ factors 
through the Kan complex $Y_{g(n)}$.      
If $n\geq 2$, an argument analogous 
to the one used in Remark~\ref{rem:examples of bifibrations} 
gives the existence of the required 
diagonal filler.  
\end{proof}

\begin{remark} 
There is a canonical functor $(\sSet)_{/A}\times (\sSet)_{/B}\to (\sSet)_{/(A\times B)}$ 
defined on objects by sending a pair $(S,T)$ to the product 
$S\times T\to A\times B$.  This functor is {\em divisible on the right and left} 
(see Section 7 of \cite{JT1}).  If $T\to B$ is an object 
of $(\sSet)_{/B}$, then we denote the functor right adjoint 
to $(-)\times T\colon (\sSet)_{/A}\to (\sSet)_{/(A\times B)}$ by 
$(-)/T\colon (\sSet)_{/(A\times B)}\to (\sSet)_{/A}$.  The functor 
$(-)/T$ is defined on objects as follows.  If $X\to A\times B$ 
is an object of $(\sSet)_{/(A\times B)}$, then $X/T$ is 
defined by the pullback diagram 
\[
\begin{tikzcd} 
X/T \arrow[r] \arrow[d] & X^T \arrow[d] \\ 
A \arrow[r] & A^T\times B^T, 
\end{tikzcd} 
\]
where the map $A\to A^T\times B^T$ is isomorphic to the product 
of the diagonal map $A\to A^T$ and the 
constant map $\Delta^0\to B^T$ given by the structure map 
$T\to B$.   
\end{remark} 

\begin{remark} 
\label{rem:canonical map from divisions}
Suppose that $X\to Y$ is a map in $(\sSet)_{/(A\times B)}$ 
and that $S\to T$ is a map in $(\sSet)_{/B}$.  
There is a canonical commutative diagram in 
$(\sSet)_{/A}$ of the form 
\[
\begin{tikzcd} 
X/T \arrow[r] \arrow[d] & Y/T \arrow[d] \\ 
X/S \arrow[r] & Y/S 
\end{tikzcd} 
\]
and an induced map 
\[
X/T \to X/S\times_{Y/S}Y/T 
\]
in $(\sSet)_{/A}$.  
\end{remark} 

The following lemma gives a sufficient 
criterion for the induced map from 
Remark~\ref{rem:canonical map from divisions} 
to be a right fibration.  

\begin{lemma} 
\label{rem:divisible exp}
Let $X\to Y$ be a bifibration in $(\sSet)_{/(A\times B)}$ and 
$S\to T$ be a monomorphism in $(\sSet)_{/B}$.  Then the 
induced map 
\[
X/T \to X/S\times_{Y/S}Y/T 
\]
is a left fibration in $(\sSet)_{/A}$.
\end{lemma}  

\begin{proof}
This follows from 
Remark~\ref{rem:restriction of 2-sided fibn}, using the 
fact that 
\[
X^T\to X^S\times_{Y^S} Y^T 
\]
is a bifibration in $(\sSet)_{/(A^T\times B^T)}$ 
by Proposition~\ref{prop:exp}.  
\end{proof} 

As an application of this lemma we have the following useful 
proposition. 

\begin{proposition} 
Let $T\to B$ be an object of $(\sSet)_{/B}$.  Then the adjoint pair 
\[
(-)\times T \colon (\sSet)_{/A}\rightleftarrows 
(\sSet)_{/(A\times B)}\colon (-)/T 
\]
is a Quillen adjunction for the covariant model 
structure on $(\sSet)_{/A}$ and the bivariant model 
structure on $(\sSet)_{/(A\times B)}$.  
\end{proposition} 

\begin{proof} 
It is clear that the functor 
$(-)/T$ preserves trivial fibrations.  Therefore 
it suffices to prove that $(-)/T$ preserves 
fibrations between fibrant objects.  Hence it suffices 
to prove that if $X\to Y$ is a bifibration 
in $(\sSet)_{/(A\times B)}$, then $X/T\to Y/T$ is a 
left fibration in $(\sSet)_{/A}$.  
This follows immediately from 
Lemma~\ref{rem:divisible exp}, taking $S = \emptyset$.  
\end{proof} 

In particular, taking $T\to B$ to be the identity map $\id_B\colon 
B\to B$, we see that the functor $\pi_A^*\colon (\sSet)_{/A}\to 
(\sSet)_{/(A\times B)}$ is left Quillen for the covariant 
model structure on $(\sSet)_{/A}$ and the bivariant model 
structure on $(\sSet)_{/(A\times B)}$.  An analogous statement 
is true for the functor $\pi_B^*$.  
We record this observation in the 
following proposition.  

\begin{proposition} 
Let $A$ and $B$ be simplicial sets.  Then the following statements are true: 
\begin{enumerate}
\item The adjunction 
\[
\pi_{B}^*\colon (\sSet)_{/B} \rightleftarrows (\sSet)_{/(A\times B)}\colon (\pi_{B})_* 
\]
is a Quillen adjunction for the bivariant model structure on $(\sSet)_{/(A\times B)}$ 
and the contravariant model structure on $(\sSet)_{/B}$; 
\item The adjunction 
\[
\pi_{A}^*\colon (\sSet)_{/A} \rightleftarrows (\sSet)_{/(A\times B)}\colon (\pi_{A})_* 
\]
is a Quillen adjunction for the bivariant model structure on $(\sSet)_{/(A\times B)}$ 
and the covariant model structure on $(\sSet)_{/A}$.  
\end{enumerate} 
\end{proposition}

\subsection{Bivariant equivalences} 
\label{sec:bivariant eqs}
In this section we establish some useful facts about bivariant equivalences.  

\begin{proposition} 
\label{prop:fibers contractible}
Let $A$ and $B$ be simplicial sets. 
Let $f\colon X\to Y$ be a bifibration in $(\sSet)_{/(A\times B)}$.  
If the fibers of $f$ are contractible 
then $f$ is a trivial Kan fibration.   
\end{proposition} 

\begin{proof} 
We need to prove that the map $f$ has the right lifting property against the inclusion 
$\partial\Delta^n\subseteq \Delta^n$ for all $n\geq 0$.  This is clear when $n=0$, since the 
fibers of $f$ are non-empty.  Suppose $n>0$.  Consider a commutative diagram of the form 
\[
\begin{tikzcd} 
\partial\Delta^n \arrow[d] \arrow[r,"\psi"] & X \arrow[d,"f"] \\ 
\Delta^n \arrow[ur,dashed] \arrow[r] & Y.
\end{tikzcd} 
\]
We want to show that the dotted arrow exists making the diagram commute.  By a base change we 
may suppose that $A = B = \Delta^n$, $Y = \Delta^n$ and that the structure map $Y\to A\times B$ 
is the diagonal inclusion $\Delta^n\to \Delta^n\times \Delta^n$.   

Let $h\colon \Delta^n\times \Delta^1\to \Delta^n$ denote 
the canonical projection.
Let $k\colon \Delta^n\times \Delta^1\to \Delta^n$ be the canonical contraction of $\Delta^n$ 
onto its final vertex so that $k|\Delta^n\times \set{0} = \mathrm{id}_{\Delta^n}$ and 
$k|\Delta^n\times \set{1}$ is the constant map on the final 
vertex.    

The inclusion $\partial\Delta^n\times \set{0}\to \partial\Delta^n\times \Delta^1$ in 
$(\sSet)_{/(A\times B)}$ with structure map $\lambda :=(h,k)|\partial\Delta^n\times \Delta^1$  
is a bivariant anodyne map in $(\sSet)_{/(A\times B)}$ by 
Proposition~\ref{prop:alt desc biv anodyne} with $K=\emptyset$, 
$L=\partial\Delta^n$.  Hence we can find the indicated diagonal filler in the diagram 
\[
\begin{tikzcd} 
\partial\Delta^n\times \set{0}\arrow[d] \arrow[r] & X \arrow[d,"f"] \\ 
\partial\Delta^n \times \Delta^1 \arrow[ur,dashed,"\phi"] \arrow[r,"\lambda"'] & Y. 
\end{tikzcd} 
\]
Observe that $\phi|\partial\Delta^n\times \set{1}$ has image inside the contractible Kan complex 
$X|\set{n}$.  Hence we may extend $\phi|\partial\Delta^n\times \set{1}$ to a map $\tilde{\phi}
\colon \Delta^n\times \set{1}\to X$.  We have a commutative diagram 
\[
\begin{tikzcd} 
\partial\Delta^n\times \Delta^1\cup \Delta^n\times \set{1}\arrow[r,"\psi'"] \arrow[d] & X \arrow[d,"f"] \\ 
\Delta^n\times \Delta^1 \arrow[r] & \Delta^n\times \Delta^n 
\end{tikzcd} 
\]
in which $\psi' = \phi\cup \tilde{\phi}$.  Observe that $\psi|\set{n}\times \Delta^1$ is an 
equivalence in $X|\set{n}$ and hence in $X$.  It follows from Proposition 2.4.1.5 and Proposition 2.4.1.8 
of \cite{HTT} that there exists a diagonal filler $\psi''\colon \Delta^n\times \Delta^1\to X$ 
for this diagram.  The restriction $\psi''|\Delta^n\times \set{0}$ is then the desired extension 
of the original map $\psi$.  
\end{proof} 

\begin{theorem}
\label{thm:pointwise biveq}
Let $A$ and $B$ be simplicial sets and let $X\to A\times B$ and $Y\to A\times B$ be bifibrations.  
Then a map $f\colon X\to Y$ in $(\sSet)_{/(A\times B)}$ is a bivariant equivalence if and only if it 
is a pointwise weak homotopy equivalence.  
\end{theorem} 

\begin{proof} 
$(\Rightarrow)$  Suppose $f\colon X\to Y$ is a bivariant equivalence between bifibrations.  
By Lemma~\ref{lem:fiberwise hty equiv}, $f\colon X\to Y$ is a fiberwise homotopy equivalence.  Therefore 
$f$ is a pointwise weak homotopy equivalence, since every fiberwise homotopy equivalence is a 
pointwise homotopy equivalence.      

$(\Leftarrow)$  Let $f\colon X\to Y$ be a pointwise weak 
homotopy equivalence.  Suppose first that $f\colon X\to Y$ is a bivariant fibration.  Then $f$ is a trivial 
fibration by Proposition~\ref{prop:fibers contractible}, since the fibers of $f$ are contractible.  
Hence $f$ is a bivariant equivalence (Lemma~\ref{lem:triv fibn biv equiv}).  Now suppose that 
$f$ is an arbitrary pointwise weak homotopy equivalence between bifibrations.  Via the small object argument, we may factor 
$f$ as $f=hg$, where $h\colon X'\to Y$ is a bifibration and where $g\colon X\to 
X'$ is a bivariant anodyne map.  Then $g$ is a bivariant equivalence 
by Lemma~\ref{lem:biv ano implies biv eq}; since it is a bivariant 
equivalence between bifibrations it is a pointwise weak homotopy equivalence 
by the forward implication proved above.  
Hence $h$ is a pointwise weak homotopy equivalence.  By 
Proposition~\ref{prop:char of bifibns between bifibns} we see that 
$h$ is a bivariant fibration.  Hence it is a bivariant equivalence  
by the special case we have proven above.  
\end{proof} 

The following characterization of bivariant 
equivalences is anticipated by Theorem~\ref{thm:7 from covariant paper}.  
This characterization is due to Joyal.     

\begin{theorem}[Joyal] 
\label{thm:char of biv eq}
Let $A$ and $B$ be simplicial sets and let $f\colon X\to Y$ be a map in $(\sSet)_{/(A\times B)}$. 
The following statements are equivalent: 
\begin{enumerate}[(i)]
\item the map $f\colon X\to Y$ is a bivariant equivalence; 

\item if $R\to A$ is a right fibration and $L\to B$ is a left fibration then the induced map 
\[
R\times_A X\times_B L\to R\times_A Y\times_B L 
\]
is a weak homotopy equivalence.  

\item for every pair of vertices $a\in A$ and $b\in B$, if $\set{a}\to Ra\to A$ is a 
factorization of $\set{a}\to A$ into a right anodyne map followed by a right fibration, and 
if $\set{b}\to Lb\to B$ is a factorization of $\set{b}\to B$ into a left anodyne map 
followed by a left fibration, then the induced map 
\[
Ra\times_A X\times_B Lb\to Ra\times_A Y\times_B Lb 
\]
is a weak homotopy equivalence; 
\end{enumerate}
\end{theorem}

\begin{proof} 
The proof that (ii) implies (iii) is trivial.  We prove that (i) implies (ii).  
Suppose that $f\colon X\to Y$ is a bivariant equivalence in $(\sSet)_{/(A\times B)}$.  
As in the proof of Proposition~\ref{prop:biv eqs access} above, we can find a commutative 
diagram in $(\sSet)_{/(A\times B)}$ of the form 
\[
\begin{tikzcd} 
X \arrow[d] \arrow[r,"f"] & Y \arrow[d] \\ 
X' \arrow[r,"f'"'] & Y' 
\end{tikzcd} 
\]
in which the vertical arrows are bivariant anodyne maps and $f'\colon X'\to Y'$ is a trivial 
Kan fibration.  It follows that without loss of generality, we may suppose that 
$f\colon X\to Y$ is a bivariant anodyne map.    

Therefore we will prove that if $f\colon X\to Y$ is a bivariant anodyne map then the induced 
map $R\times_A X\times_B L\to R\times_A Y\times_B L$ is a weak homotopy equivalence.  
The class of all maps $X\to Y$ in $(\sSet)_{/(A\times B)}$ with this property is weakly saturated.  
Therefore it suffices to show that this class contains all maps of the form (1), (2) 
and (3) from Definition~\ref{def:biv anodyne}.  It is clear, using 
Proposition 3.3.1.3 from \cite{HTT} that this class contains all maps of the form (1).  
By a duality argument, it 
suffices to prove this for all maps of the form (2) from Definition~\ref{def:biv anodyne}.  

By a base-change argument we may suppose 
that $A = \Delta^n$.  Let $R\to A$ be a right fibration.  
We will prove the following statement: if $\Lambda^n_0\to \Delta^n$ 
is a map of the form (2) in $(\sSet)_{/(A\times B)}$ 
from Definition~\ref{def:biv anodyne} then the image of $R\times_A \Lambda^n_0\to 
R\times_A \Delta^n$ under $(\pi_B)_!\colon (\sSet)_{/(R\times B)}\to 
(\sSet)_{/B}$ is a contravariant equivalence in $(\sSet)_{/B}$.  Applying 
Theorem~\ref{thm:7 from covariant paper}, we deduce that $R\times_A \Delta^{\set{0}}\times_B L\to 
R\times_A \Delta^1\times_B L$ is a weak homotopy equivalence.

We use the theory of mapping simplexes (see 
Section 3.2.2 of \cite{HTT}).  There is a sequence $\phi\colon A^n\to \cdots \to A^1\to A^0$ 
of composable morphisms between Kan complexes and a quasi-equivalence $M(\phi)\to R$ (see Definition 3.2.2.6 of 
\cite{HTT}).  We have a pullback diagram 
of the form 
\[
\begin{tikzcd} 
M(\phi)|_{\Lambda^n_0} \arrow[r] \arrow[d] & R|_{\Lambda^n_0} \arrow[d] \\ 
M(\phi) \arrow[r] & R 
\end{tikzcd} 
\]
in which the horizontal maps are categorical equivalences 
by Proposition 3.2.2.10 of \cite{HTT}.  
Therefore, it suffices to prove that $(\pi_B)_!M(\phi)|\Lambda^n_0\to (\pi_B)_!M(\phi)$ is a 
contravariant equivalence in $(\sSet)_{/B}$, since every categorical equivalence 
is a contravariant equivalence.  
But the map $M(\phi)|\Lambda^n_0 \to M(\phi)$ forms part of a pushout diagram 
\begin{equation} 
\label{eq:mapping simplex pushout}
\begin{tikzcd} 
A^n\times \Lambda^n_0 \arrow[d] \arrow[r] & M(\phi)\times_{\Delta^n}\Lambda^n_0 \arrow[d] \\ 
A^n\times \Delta^n \arrow[r] & M(\phi) 
\end{tikzcd} 
\end{equation}
and hence is bivariant anodyne, since it is the pushout of the bivariant anodyne 
map $A^n\times \Lambda^n_0 \to A^n\times \Delta^n$.  This suffices to complete the proof, 
by (1) of Proposition~\ref{prop:biv eqs and co and cont eqs}.    

To see that the diagram~\eqref{eq:mapping simplex pushout} above is a pushout, observe that from the proof of 
Proposition 3.2.2.10 from \cite{HTT} we have a pushout diagram of the form 
\[
\begin{tikzcd} 
A^n\times \Delta^{\set{0,\ldots,n-1}} \arrow[r] \arrow[d] & M(\phi') \arrow[d] \\ 
A^n\times \Delta^n \arrow[r] & M(\phi) 
\end{tikzcd} 
\]
where $\phi'$ denotes the composable sequence 
$\phi'\colon A^{n-1}\to \cdots \to A^1\to A^0$.  It follows 
that the top square and the outer square in the composite diagram 
\[
\begin{tikzcd} 
A^n\times \Delta^{\set{0,\ldots,n-1}} \arrow[d] \arrow[r] & 
M(\phi')\times_{\Delta^n}\Lambda^n_0 \simeq M(\phi') \arrow[d] \\ 
A^n\times \Lambda^n_0 \arrow[r] \arrow[d] & M(\phi)\times_{\Delta^n}\Lambda^n_0 \arrow[d] \\ 
A^n\times \Delta^n \arrow[r] & M(\phi) 
\end{tikzcd} 
\]
are pushouts, and hence so is the diagram~\eqref{eq:mapping simplex pushout}.

Finally, suppose that (iii) holds; we will prove that (i) holds, i.e.\ $f$ is a bivariant equivalence.  
Via the small object argument, we may find a commtutative diagram 
\[
\begin{tikzcd} 
X \arrow[r,"f"] \arrow[d]{} & Y \arrow[d] \\ 
X' \arrow[r] & Y' 
\end{tikzcd} 
\]
in which the vertical maps are bivariant 
anodyne maps and $X',Y'$ are bifibrations 
with structure maps $(p_{X'},q_{X'})
\colon X'\to A\times B$, $(p_{Y'},q_{Y'})\colon Y'\to A\times B$ respectively.  
By Theorem~\ref{thm:pointwise biveq} it suffices 
to show that $X'\to Y'$ is a pointwise homotopy equivalence.  Let $a\in A$ 
and $b\in B$ be vertices, and let $\set{a}\to Ra\to A$, 
$\set{b}\to Lb\to B$ be factorizations of 
$\set{a}\to A$ and $\set{b}\to B$ into a right anodyne 
map followed by a right fibration, 
and a left anodyne map followed by a left fibration 
respectively.  We claim that the canonical maps 
\[
\set{a}\times_A X'\times_B \set{b}\to Ra\times_A X'\times_B Lb 
\]
and 
\[
\set{a}\times_A Y'\times_B \set{b}\to Ra\times_A Y'\times_B Lb 
\]
are weak homotopy equivalences.  The first map above factors 
as 
\[
\set{a}\times_A X'\times_B \set{b}\to Ra\times_A X'\times_B \set{b}\to Ra\times_A X'\times_B Lb 
\]
The map $\set{a}\times_A X'\times_B \set{b}\to Ra\times_A X'\times_B \set{b}$ 
is right anodyne (and hence a weak homotopy equivalence) since 
$X'\times_{B} \set{b}\to A\times \set{b}$ is a bifibration 
and hence the composite map $X'\times_B \set{b}\to 
A\times \set{b}\to A$ 
is smooth (Proposition 4.1.2.15 of \cite{HTT}).  
Similarly, $Ra\times_{A}X'\to Ra\times B$ is a 
bifibration, and hence a duality argument 
using Proposition 4.1.2.15 of \cite{HTT} again 
shows that the induced map $Ra\times_{A} X'\times_{B} \set{b}\to Ra\times_{A} 
X'\times_{B} Lb$ is left anodyne (and hence is a 
weak homotopy equivalence).  It follows that the first 
canonical map above is a weak homotopy equivalence.  The proof 
that the second canonical map above is a weak homotopy 
equivalence is completely analogous.      

Therefore, under the hypothesis that (iii) holds, we see that $X'\to Y'$ is a pointwise weak 
homotopy equivalence (and hence $f$ is a bivariant equivalence) if and only if the 
two vertical maps 
\[
Ra\times_A X\times_B Lb\to Ra\times_A X'\times_B Lb\quad \text{and}\quad 
Ra\times_A Y\times_B Lb\to Ra\times_A Y'\times_B Lb 
\]
are weak homotopy equivalences.  This follows from the implication (i) implies (ii), 
which we have already proven.    
\end{proof} 

The following is a very useful example of a 
bivariant equivalence.  

\begin{lemma} 
\label{lem:diag biv anodyne}
Let $A$ and $B$ be simplicial sets.  
For every $n\geq 0$ and for every map 
$\Delta^n\to A\times B$ in $(\sSet)_{/(A\times B)}$, 
the diagonal map $\Delta^n\to \Delta^n\times 
\Delta^n$ in $(\sSet)_{/(A\times B)}$ 
is a bivariant 
equivalence.  
\end{lemma} 

\begin{proof} 
We use the characterization of bivariant equivalences 
from Theorem~\ref{thm:char of biv eq}.  Let $a\in A$ 
and $b\in B$ be vertices.  We have a commutative diagram 
\[
\begin{tikzcd} 
A_{/a}\times_A \Delta^n\times_B B_{b/} \arrow[r] \arrow[d] & 
A_{/a}\times_A \Delta^n \times \Delta^n\times_{B} B_{b/} \arrow[d] \\ 
A_{/a}\times_A \Delta^n \arrow[r] \arrow[d] & 
A_{/a}\times_A \Delta^n \times \Delta^n \arrow[d]    \\ 
\Delta^n \arrow[r] & \Delta^n\times \Delta^n 
\end{tikzcd} 
\]
in which both squares are pullbacks.  
It suffices to prove that the middle horizontal map is 
right anodyne, since the base change of a right anodyne map 
along a left fibration is right anodyne, and hence a weak 
homotopy equivalence.  Since a dominant map is both 
left and right cofinal (Remark~\ref{rem:dom and cofinal}), and dominant 
maps are preserved under base change along right fibrations 
(Lemma~\ref{lem:dom maps closed under cobase change along lf}), 
the result follows from the fact that $\Delta^n\to \Delta^n\times \Delta^n$ 
is dominant (Lemma~\ref{lem:diagonal dominant}).  
\end{proof} 

\begin{remark} 
Note that in general a map in $(\sSet)_{/(A\times B)}$ 
of the form $\Delta^n\to \Delta^n\times \Delta^n$ is not 
bivariant anodyne.  
\end{remark} 

\begin{proposition} 
Suppose $X\to Y$ is a categorical equivalence in $(\sSet)_{/(A\times B)}$.  Then $X\to Y$ is a 
bivariant equivalence.  
\end{proposition} 

\begin{proof} 
We use Theorem~\ref{thm:char of biv eq}.  Let $R\to A$ be a right fibration and let $L\to B$ be a 
left fibration.  Since $X\to Y$ is a categorical equivalence, it follows that $R\times_A X\to 
R\times_A Y$ is a categorical equivalence by Proposition 3.3.1.3 of \cite{HTT}. 
The induced map $R\times_A Y\times_B L\to R\times_A Y$ 
is a left fibration and hence $R\times_A X\times_B L\to R\times_A Y\times_B L$ is a 
categorical equivalence by Proposition 3.3.1.3 of \cite{HTT} again.  Hence the map 
$R\times_A X\times_B L\to R\times_AY\times_{B}L$ is a weak homotopy equivalence.  
\end{proof} 

The following corollaries are straightforward and are 
left to the reader.  

\begin{corollary} 
Let $A$ and $B$ be simplicial sets and let $(p,q)\colon X\to A\times B$ be a bifibration.  Then 
$(p,q)$ is a categorical fibration.  
\end{corollary} 

\begin{corollary} 
The bivariant model structure on $(\sSet)_{/(A\times B)}$ is a left Bousfield localization 
of the Joyal model structure on $(\sSet)_{/(A\times B)}$. 
\end{corollary}

\begin{remark} 
We could have obtained the bivariant 
model structure  
by taking a left Bousfield localization of 
the Joyal model structure on $(\sSet)_{/(A\times B)}$ 
at the set of horn inclusions 
$\Lambda^n_0\subseteq \Delta^n$ 
and $\Lambda^n_n\subseteq \Delta^n$ 
in $(\sSet)_{/(A\times B)}$ 
of the form described in Definition~\ref{def:bifibration}.  
However, this approach would require us 
to prove that every bifibration 
is a categorical fibration.  
This is straightforward to prove if 
$A$ and $B$ are $\infty$-categories, 
but it is not a priori
obvious for arbitrary simplicial sets 
$A$ and $B$.  
\end{remark} 

\subsection{Quillen equivalences} 
\label{subsec:quillen equivs}
Recall that there is another subdivision 
functor for simplicial sets introduced in 
\cite{BHM}.  This functor $\sd_2\colon 
\sSet\to \sSet$ is defined by composing  
a functor $X\colon \Delta^{\op}\to \Sets$ 
with the {\em diagonal} functor 
\[ 
\delta\colon \Delta\to \Delta,\quad   
[n] \mapsto [n]\star [n].  
\] 
Just as in the earlier case for Segal's 
subdivision functor in 
Section~\ref{subsec:subdivision}, the  
functor $\sd_2$ can be used to relate the category 
$(\sSet)_{/(A\times B)}$ with the 
category of correspondences from $A$ to $B$.  
The functor $\delta$ induces a functor 
\[
\delta\colon \Delta_{/(A\times B)}\to (\sSet)_{/B\star A}  
\]
which sends a pair $(u,v)\colon \Delta^n\to A\times B$ to  
\[
\delta(u,v) = v\star u\colon \Delta^n\star \Delta^n\to B\star A.  
\]
The functor $\delta$ induces an adjunction 
\[
\delta_!\colon (\sSet)_{/(A\times B)}\rightleftarrows (\sSet)_{/B\star A}\colon \delta^* 
\]
and in fact the functor $\delta^*$ has a further right adjoint 
$\delta_*\colon (\sSet)_{/(A\times B)}\to (\sSet)_{/B\star A}$.

Composing the functor $\delta_!$ with the reflector $L\colon (\sSet)_{/B\star A}\to 
\Corr(A,B)$ (see Remark~\ref{rem:reflector L}) gives rise to a 
functor denoted 
$d_!\colon (\sSet)_{/(A\times B)}\to \Corr(A,B)$.  We have an adjoint pair 
\[
d_!\colon (\sSet)_{/(A\times B)}\rightleftarrows \Corr(A,B)\colon d^*.  
\]
Similarly we have an adjoint pair 
\begin{equation} 
\label{eq:d* quillen adjn}
d^*\colon \Corr(A,B)\rightleftarrows (\sSet)_{/(A\times B)} \colon d_* 
\end{equation}
where $d_*$ is the functor which sends an 
object $X\to A\times B$ in $(\sSet)_{/(A\times B)}$ to 
the correspondence $d_*X$ whose set of $n$-simplices 
is 
\[
(d_*X)_n = (\sSet)_{/(A\times B)}(d^*L(\Delta^n),X) 
\]
where $\Delta^n\to B\star A$ is an $n$-simplex.  

\begin{remark}
\label{rem:delta*}
Analogous to Lemma~\ref{lem:sigma*}, if $f\colon C\to A$ and $g\colon D\to B$ are maps 
of simplicial sets determining objects $D\star C$ and $C\times D$ of $(\sSet)_{/B\star A}$ and 
$(\sSet)_{/(A\times B)}$ respectively, then we have $\delta^*(D\star C) = C\times D$.  
Note also that, analogous to Remark~\ref{rem:calculation of unit map}, 
the functor $\delta^*\colon (\sSet)_{/B\star A}\to (\sSet)_{/(A\times B)}$ sends 
the object $D\sqcup C$ to the initial object $\emptyset$ of 
$(\sSet)_{/(A\times B)}$.  It follows that there is 
a natural isomorphism of functors 
$d^*d_!\simeq \delta^*\delta_!$.     
\end{remark}

\begin{remark} 
\label{rem:BHM subdivision}
The relationship between the subdivision functor $\sd_2$ from 
\cite{BHM} and the functor $d^*$ is as follows.  If $X\in \Corr(A,B)$ 
then there is a pullback diagram of the form 
\[
\begin{tikzcd} 
d^*X \arrow[rr] \arrow[d] & &  \sd_2 X \arrow[d] \\ 
A\times B \arrow[r] & B\times A \arrow[r] & X\times X 
\end{tikzcd} 
\]
where the map $A\times B\to B\times A$ is the switch map 
which interchanges the two factors and where the 
map $B\times A\to X\times X$ is induced by the inclusions 
$B\subseteq X$ and $A\subseteq X$.  
This is analogous to the relationship between $a^*X$ and the 
twisted arrow category, or Segal edgewise subdivision of $X$ 
(see Remark~\ref{rem:alt desc of a*}).
\end{remark} 

\begin{remark} 
\label{rem:geom realization}
Recall (see Lemma 1.1 of \cite{BHM} and 
Proposition (A.1) of \cite{Segal}) that 
for any simplicial set $X$ there are natural  
isomorphisms  
$|\sd_2 X|\simeq |X|$ and $|\Tw(X)|\simeq |X|$ 
on geometric realizations.  In particular there is 
an isomorphism $|\Tw(X)|\simeq |\sd_2 X|$, natural in $X$.  Recall also 
that there is a canonical isomorphism $|X^{\op}|\simeq |X|$
between the geometric realization of a simplicial set, 
and the geometric realization of the opposite simplicial set.  
We claim that the following diagram 
commutes 
\[
\begin{tikzcd} 
{|}\Tw(X){|} \arrow[r] \arrow[d] & {|}X^{\op}{|}\times {|}X{|} \arrow[d] \\ 
{|}\sd_2 X{|} \arrow[r] & {|}X{|}\times {|}X{|} 
\end{tikzcd} 
\]
where the left hand vertical map is the isomorphism mentioned 
above, and the right hand vertical map is the product of   
the canonical isomorphism $|X^{\op}|\simeq |X|$ and the 
identity map on $|X|$. The map $|\Tw(X)|\to |X^{\op}|\times |X|$ 
is induced by the inclusions $[n]^{\op}\subseteq [n]^{\op}\star [n]$ 
and $[n]\subseteq [n]^{\op}\star [n]$.  Similarly the map 
$|\sd_2 X|\to |X|\times |X|$ is induced by the two canonical 
inclusions $[n]\subseteq [n]\star [n]$.  

Since all of the functors involved commute with colimits, it 
suffices by naturality to prove the claim in the special case 
when $X = \Delta^n$.  Since all of the functors involved also 
commute with finite products, and $\Delta^n$ is a retract 
of $(\Delta^1)^n$, it suffices to prove the statement 
when $X = \Delta^1$.   

Under the isomorphism $|\Tw(\Delta^1)|\to |\Delta^1|\simeq [0,1]$ 
(Proposition (A.1) of \cite{Segal}), the 
induced map $|\Tw(\Delta^1)|\to |(\Delta^1)^{\op}|\times |\Delta^1|$ 
corresponds to the map $(f,g)\colon [0,1]\to [0,1]\times [0,1]$ 
where 
\[
f(t) = \begin{cases} 
1-2t & \text{if}\ 0\leq t\leq 1/2, \\ 
0 & \text{if}\ 1/2\leq t\leq 1, 
\end{cases} 
\quad \text{and} \quad 
g(t) = \begin{cases} 
0 & \text{if}\ 0\leq t\leq 1/2, \\ 
2t-1 & \text{if}\ 1/2\leq t\leq 1.  
\end{cases} 
\]
Under the isomorphism $|\sd_2 \Delta^1|\to 
|\Delta^1|\simeq [0,1]$ (Lemma 1.1 of \cite{BHM}), the induced maps 
$|\sd_2 \Delta^1|\to |\Delta^1|\times |\Delta^1|$ corresponds to the 
map $(f',g')\colon [0,1]\to [0,1]\times [0,1]$ where 
\[
f'(t) = \begin{cases} 
2t & \text{if}\ 0\leq t\leq 1/2, \\ 
1 & \text{if}\ 1/2\leq t\leq 1, 
\end{cases} 
\quad \text{and} \quad 
g'(t) = \begin{cases} 
0 & \text{if}\ 0\leq t\leq 1/2, \\ 
2t-1 & \text{if}\ 1/2\leq t\leq 1.  
\end{cases} 
\] 
Under the identification $|\Delta^1|\simeq [0,1]$, 
the isomorphism $|(\Delta^1)^{\op}|\to |\Delta^1|$ 
corresponds to the automorphism of $[0,1]$ which sends 
$t$ to $1-t$.  Clearly the composite of $f$ with this 
automorphism is equal to $f'$.  The claim follows.      
\end{remark} 

Our first 
aim is to prove that the adjunction~\eqref{eq:d* quillen adjn} 
is a Quillen adjunction.  
We have the following result.  

\begin{proposition} 
\label{prop:d^* d_* Quillen adjunction}
Let $A$ and $B$ be $\infty$-categories.  Then the adjunction 
\[
d^*\colon \Corr(A,B)\rightleftarrows (\sSet)_{/(A\times B)}
\colon d_* 
\]
is a Quillen adjunction for the correpondence model structure 
on $\Corr(A,B)$ and the bivariant model structure on 
$(\sSet)_{/(A\times B)}$.  
\end{proposition} 

\begin{proof} 
Proving the proposition quickly reduces to proving the following claim: the 
functor $\delta^*\colon (\sSet)_{/B\star A}\to (\sSet)_{/(A\times B)}$ 
sends inner anodyne maps in $(\sSet)_{/B\star A}$ to 
bivariant anodyne maps in $(\sSet)_{/(A\times B)}$.  
Recall from the proof of Theorem~\ref{thm: (delta_!,d) is a Quillen equivalence} 
that the inner anodyne maps in $(\sSet)_{/B\star A}$ 
are of the following form: 
\begin{itemize} 
\item $\Lambda^m_k\star \emptyset\to \Delta^m\star \emptyset$, $0<k<n$, 
\item $\Delta^m\star \Lambda^n_k\cup \partial\Delta^m\star \Delta^n \to \Delta^m\star \Delta^{n}$, $m\geq 0$, $0\leq k<n$, 
\item $\Lambda^m_k\star \Delta^{n}\cup \Delta^m\star \partial\Delta^{n}\to \Delta^m\star \Delta^{n}$, $0<k\leq m$, $n\geq 0$,  
\item $\emptyset\star \Lambda^n_k\to \emptyset\star \Delta^n$, $0<k<n$,  
\end{itemize} 
for some simplices $x\colon \Delta^m\to B$ and $y\colon \Delta^n\to A$.  
It follows from Remark~\ref{rem:delta*} that 
it suffices to prove the following two statements: 
the canonical map 
\[
\Lambda^n_k\times \Delta^m\cup  
\Delta^n\times \partial\Delta^m\to  \Delta^n\times \Delta^m
\]
is a bivariant anodyne map in $(\sSet)_{/(A\times B)}$ 
if $m\geq 0$ and $0\leq k<n$; and 
the canonical map  
\[
\Delta^n\times \Lambda^m_k \cup  
\partial\Delta^n \times \Delta^m\to  \Delta^n \times\Delta^m
\]
is a bivariant anodyne map in $(\sSet)_{/(A\times B)}$ 
if $0<k\leq m$ and $n\geq 0$.    
The first statement follows readily from Lemma~\ref{rem:divisible exp} 
and the second statement follows by duality.  
\end{proof} 

Our next aim is to prove that the Quillen adjunction 
from Proposition~\ref{prop:d^* d_* Quillen adjunction} 
is in fact a Quillen equivalence.  
We first need a preliminary result.  

\begin{lemma} 
\label{lem:d^*d_* triv kan fibn}
Let $A$ and $B$ be simplicial sets and 
let $X\to A\times B$ be a bifibration 
in $(\sSet)_{/(A\times B)}$.  Then the counit 
map $d^*d_*X\to X$ is a trivial Kan fibration.  
\end{lemma} 

\begin{proof} 
By adjointness it suffices to prove 
that the induced map 
\[
d^*d_!\partial\Delta^n\cup 
{\partial\Delta^n}\to 
d^*d_!\Delta^n 
\]
is an acyclic cofibration in the 
bivariant model structure 
for every boundary inclusion 
$\partial\Delta^n\subseteq \Delta^n$ 
in $(\sSet)_{/(A\times B)}$.  

The unit map 
$\Delta^n\to d^*d_!\Delta^n$ is a bivariant 
equivalence (Lemma~\ref{lem:diag biv anodyne}).  Therefore, by a
2-out-of-3 argument, it suffices to prove 
that the unit map $S\to d^*d_!S$ is a monic 
bivariant equivalence for every object 
$S\to A\times B$ in $(\sSet)_{/(A\times B)}$.  
Recall (Remark~\ref{rem:delta*}) that 
there is an isomorphism $d^*d_!S = \delta^*\delta_!S$.    

Using the skeletal filtration of $S$ we see that 
by an induction argument we are reduced to the 
case where $S$ is obtained from $S'$ by adjoining 
a single $n$-simplex along an attaching map 
$\partial\Delta^n\to S'$ in $(\sSet)_{/(A\times B)}$.  
We have a commutative diagram in 
$(\sSet)_{/(A\times B)}$ of the form 
\[
\begin{tikzcd} 
\Delta^n \arrow[d] & \arrow[l]
\partial\Delta^n \arrow[d] \arrow[r] & S' \arrow[d] \\ 
\delta^*\delta_! \Delta^n & \arrow[l] 
\delta^*\delta_!\partial\Delta^n \arrow[r] 
& \delta^*\delta_!S' 
\end{tikzcd} 
\]
in which the middle and right hand vertical maps 
are monic bivariant equivalences by the induction 
hypothesis.  
A straightforward argument, using the fact that 
the bivariant model structure is left proper, 
shows that the induced map 
\[
S = \Delta^n\cup_{\partial\Delta^n}S' \to 
\delta^*\delta_!S = \delta^*\delta_!\Delta^n 
\cup_{\delta^*\delta_!\partial\Delta^n} 
\delta^*\delta_!S' 
\]
is a bivariant equivalence.  To close the inductive 
loop we need to prove that $S\to d^*d_!S$ is monic.  
For this it suffices to prove that for any $n\geq 0$ 
the square 
\[
\begin{tikzcd} 
\partial\Delta^n \arrow[r] \arrow[d] & 
\delta^*\delta_!\partial\Delta^n \arrow[d] \\ 
\Delta^n \arrow[r] & \delta^*\delta_!\Delta^n 
\end{tikzcd} 
\]
is a pullback in $(\sSet)_{/(A\times B)}$.  The proof 
of this is completely analogous to the proof 
of the corresponding fact in the proof of 
Theorem~\ref{thm: (delta_!,d) is a Quillen equivalence} 
and is omitted.  
\end{proof} 

\begin{theorem}
\label{thm:d^* d_* QE}
Let $A$ and $B$ be $\infty$-categories.  
The Quillen adjunction 
\[
d^*\colon \Corr(A,B)\rightleftarrows (\sSet)_{/(A\times B)} \colon 
d_* 
\]
of Proposition~\ref{prop:d^* d_* Quillen adjunction} 
is a Quillen equivalence.  
\end{theorem} 

\begin{proof} 
From Lemma~\ref{lem:d^*d_* triv kan fibn} we have 
that the counit map $d^*d_*X\to X$ is a trivial 
Kan fibration whenever $X\to A\times B$ is a 
bifibration.  Therefore it suffices to 
prove that $d^*$ reflects weak equivalences.  
Suppose then that $X\to Y$ is a map in 
$\Corr(A,B)$ such that the image 
$d^*X\to d^*Y$ is a bivariant equivalence.  
Therefore, by Theorem~\ref{thm:char of biv eq}, 
we have that the induced map $d^*X\times_{A\times B} 
(A_{/a}\times B_{b/})\to 
d^*Y\times_{A\times B}(A_{/a}\times B_{b/})$ is 
a weak homotopy equivalence for all vertices $a\in A$ 
and $b\in B$.  From Remark~\ref{rem:BHM subdivision}  
we see that there is an isomorphism 
\[
d^*X\times_{A\times B}(A_{/a}\times B_{b/}) = 
\sd_2 X\times_{X\times X}(A_{/a}\times B_{b/}) 
\]
Using Remark~\ref{rem:geom realization} 
together with the fact that there is 
an isomorphism $(B_{b/})^{\op}\simeq (B^{\op}_{/b})^{\op}$ we 
see that there is an isomorphism 
\[
|d^*X\times_{A\times B}(A_{/a}\times B_{b/})| 
\simeq |a^*X|\times_{|B^{\op}|\times |A|} 
(|B^{\op}_{/b}|\times |A_{/a}|)|
\]
natural in the correspondence $X$.    
It follows that for any vertices $a\in A$ and 
$b\in B$, the induced map 
\[
d^*X\times_{A\times B} 
(A_{/a}\times B_{b/})\to 
d^*Y\times_{A\times B}(A_{/a}\times B_{b/})
\]
is a weak homotopy equivalence if and only 
if the induced map 
\[
a^*X\times_{B^{\op}\times A}(B^{\op}_{/b}\times A_{/a}) 
\to a^*Y\times_{B^{\op}\times A}(B^{\op}_{/b}\times A_{/a}) 
\]
is a weak homotopy equivalence.  Therefore, $d^*X\to d^*Y$ 
is a bivariant equivalence if and only if 
$a^*X\to a^*Y$ is a covariant equivalence, using 
Theorem~\ref{thm:7 from covariant paper}.  Therefore $X\to Y$ is a weak 
equivalence in the correspondence model structure 
since $a^*$ reflects weak equivalences 
(Theorem~\ref{thm: (delta_!,d) is a Quillen equivalence}).
\end{proof} 

\begin{remark} 
Unlike the case for the adjoint pair $(a_!,a^*)$, the 
adjoint pair $(d_!,d^*)$ is not a Quillen adjunction.  
It can be shown (with some work) that the functor 
$d_!\colon (\sSet)_{/(A\times B)}\to \Corr(A,B)$ sends 
maps of the form (2) and (3) from Definition~\ref{def:biv anodyne} 
to inner anodyne maps in $\Corr(A,B)$, but it does 
not in general send inner horn inclusions in $(\sSet)_{/(A\times B)}$ 
to inner anodyne maps in $\Corr(A,B)$.  
\end{remark} 

Finally let us describe another 
Quillen equivalence relating the correspondence model 
structure with the 
bivariant model structure.  
Recall from \cite{AF} the functor 
\[
\Gamma\colon \Corr(A,B)\to (\sSet)_{/(A\times B)} 
\]
which sends a correspondence $X\in \Corr(A,B)$ to 
its simplicial set of sections  
$\Gamma(X) = \map_{\Delta^1}(\Delta^1,X)$.  
The structure map $\Gamma(X)\to A\times B$ 
is induced by the inclusion $\partial\Delta^1
\subseteq \Delta^1$.  The functor $\Gamma$ 
has a left adjoint $C\colon (\sSet)_{/(A\times B)} 
\to \Corr(A,B)$ which sends an 
object $(f,g)\colon X\to A\times B$ in $(\sSet)_{/(A\times B)}$ 
to the correspondence 
\[
C(X) = (X\times \Delta^1)\cup_{X\times \partial\Delta^1}(B\times A) 
\]    
where the map $X\times \partial\Delta^1\to A\times B$ restricts 
to $g$ on $X\times \set{0}$ and restricts to $f$ on $X\times \set{1}$.  

We then have the following result 
from \cite{AF}.  

\begin{theorem}[\cite{AF}] 
\label{thm:AF thm}
Let $A$ and $B$ be $\infty$-categories.  The adjoint pair 
\[
C\colon (\sSet)_{/(A\times B)} 
\rightleftarrows \Corr(A,B)\colon \Gamma 
\]
is a Quillen equivalence for the bivariant 
model structure on $(\sSet)_{/(A\times B)}$ 
and the correspondence model structure on 
$\Corr(A,B)$.  
\end{theorem} 

We only sketch the proof, since 
an $\infty$-categorical version  can be 
found in \cite{AF}.  

\begin{proof}[Sketch of proof]
The proof that $\Gamma$ is a right Quillen functor is a 
straightforward modification of the proof of 
Proposition 2.4.7.10 of \cite{HTT}.  It can be 
shown that the functor $C$ reflects 
weak equivalences (see \cite{AF}).  The result then 
follows from Proposition B.3.17 of \cite{HA}.  
\end{proof} 

\subsection*{Acknowledgements} I thank Thomas Nikolaus for some 
useful conversations.

\end{document}